\documentclass[reqno]{amsart}            
\usepackage{amsthm,amssymb,amsfonts,graphicx,cite,color}
\usepackage[bookmarks,bookmarksnumbered,bookmarksopen,colorlinks,backref,linkcolor=blue,citecolor=red]{hyperref}%
\usepackage{xcolor}
\usepackage{graphicx}
\usepackage{mathrsfs}

\newtheorem{theorem}{Theorem}[section]

\newtheorem{lemma}{Lemma}[section]
\newtheorem{corollary}{Corollary}[section]
\newtheorem{remark}{Remark}[section]
\newtheorem{proposition}{Proposition}[section]
\numberwithin{equation}{section}

\begin{document}

\title[A convection-dominated transport problem in a thin graph-like junction]
{Asymptotic expansion for   convection-dominated transport in a thin graph-like junction}
\author[Taras Mel'nyk \& Christian Rohde]{Taras Mel'nyk$^\flat$ \ \& \ Christian Rohde$^\natural$}
\address{\hskip-12pt  $^\flat$ Department of Mathematical Physics, Faculty of Mathematics and Mechanics\\
Taras Shevchenko National University of Kyiv\\
Volodymyrska str. 64,\ 01601 Kyiv,  \ Ukraine
}
\email{melnyk@imath.kiev.ua}

\address{\hskip-12pt  $^\natural$ Institute of Applied Analysis and Numerical Simulation,
Faculty of Mathematics and Physics, Suttgart University\\
Pfaffenwaldring 57,\ 70569 Suttgart,  \ Germany
}
\email{christian.rohde@mathematik.uni-stuttgart.de }

\begin{abstract}
We consider  for a small parameter $\varepsilon >0$  a parabolic  convection-diffusion problem with
P\'eclet number  of  order $\mathcal{O}(\varepsilon^{-1})$  in a three-dimensional graph-like junction consisting  of  thin curvilinear cylinders with radii of order $\mathcal{O}(\varepsilon)$ connected through a domain (node) of diameter $\mathcal{O}(\varepsilon).$  Inhomogeneous  Neumann type boundary conditions, that involve convective and diffusive fluxes, are prescribed both on the lateral surfaces of the thin cylinders and the boundary of the node.

The asymptotic behaviour of  the solution is studied as $\varepsilon \to 0,$ i.e., when  the diffusion coefficients are eliminated and the thin junction is shrunk into a three-part graph connected in a single vertex. A rigorous procedure for the construction of the complete asymptotic expansion of the solution is developed
and the corresponding energetic and uniform pointwise estimates are proven. Depending on the directions of the limit convective fluxes, the corresponding  limit problems
$(\varepsilon = 0)$  are derived in the form of  first-order hyperbolic differential equations on the one-dimensional branches with novel  gluing conditions  at the vertex.  These generalize the classical Kirchhoff transmission conditions and might require the solution of a three-dimensional cell-like problem associated with the vertex to account for the local geometric inhomogeneity of the node and the physical processes in the node.
The asymptotic ansatz consists of three parts, namely, the regular part, node-layer part, and boundary-layer one. Their coefficients are classical solutions to mixed-dimensional limit problems. The existence and other properties of those solutions are analyzed.
\end{abstract}

\keywords{Asymptotic expansion, convection-diffusion problem, asymptotic estimate, thin graph-like junction, hyperbolic limit model.
\\
\hspace*{9pt} {\it MOS subject classification:} \   35K20,  35R02,  35B40, 35B25, 	35B45, 35K57, 	35Q49
}

\maketitle
\tableofcontents
\section{Introduction}\label{Sect1}
Convection-diffusion problems are ubiquitous in various, seemingly different applications: these include effective species transport  in heterogeneous or porous  media,
 modeling semiconductor devices by  drift-diffusion equations,  the  Black--Scholes dynamics  from financial modeling, heat transfer processes, image processing in computer graphics and medical applications, or  material crystallizations (see e.g.~\cite{Mor96,Stynes-2018} and references therein).

Especially, when modeling transport phenomena in a continuous medium, much attention of researchers has been attracted by problems with the predominance of convection terms over diffusion terms \cite{Joh_Kno_Nov_2018,Stynes-2018}.
In this case, the P\'eclet number as the ratio of  convective to  diffusive transport rates  is large. In the corresponding transport model this regime is controlled by a small parameter $\varepsilon>0$, scaling the  diffusion operator. From a mathematical and numerical point of view, such problems pose severe challenges. Therefore, much research existing in the literature is limited to the study of  steady-state problems, and even in that case, the results are not complete \cite{Joh_Kno_Nov_2018,Stynes-2018}. Notably, the  derivation  of  $\varepsilon$-independent a-priori estimates   and  corresponding error estimates  is still a widely-open topic, in particular for the numerical analysis.

In this paper, we are interested in the asymptotic analysis (as $\varepsilon \to 0)$ of  the linear initial-boundary value problem  in three spatial dimensions for the parabolic equation
\begin{equation}\label{in_1}
  \partial_t u_\varepsilon -  \varepsilon\, \mathrm{div}_x\big( \mathbb{D}_\varepsilon(x) \nabla_x u_\varepsilon\big) +
  \mathrm{div}_x\big( \overrightarrow{V_\varepsilon}(x) \, u_\varepsilon\big) = 0.
\end{equation}
In \eqref{in_1}, the function  $u_\varepsilon$  denotes  the unknown concentration of some transported species.  Thereby $\overrightarrow{V_\varepsilon}$ is the given
convective  velocity field,   and    the  matrix $\mathbb{D}_\varepsilon$  governs the  diffusion  process.  They depend in a special way on the small parameter $\varepsilon$ to account for small-scale heterogeneities and for the  geometry  of  the considered domain:
we consider the  problem \eqref{in_1} in a  three-dimensional graph-like junction, consisting  of  three thin curvilinear cylinders of diameter $\mathcal{O}(\varepsilon)$ that are connected through a domain (node)  (see Fig.~\ref{f3}). In each cylinder, the diffusive and the convective transport depend on a fast, small-scale variable and  are assumed to be of  first  order with respect to $\varepsilon$ in the lateral directions.  (see Sect.~\ref{Sec:Statement}, Assumptions \textbf{A1, A2} for the exact setting).
Appropriate initial conditions and  the inhomogeneous   flux boundary conditions
\begin{equation}\label{in_2}
\big(-  \varepsilon \,  \mathbb{D}_\varepsilon(x)\nabla u_\varepsilon +  u_\varepsilon \, \overrightarrow{V_\varepsilon}(x)\big) \cdot \boldsymbol{\nu}_\varepsilon   =  \varepsilon \, \varphi^{(i)}_\varepsilon(x,t)
\end{equation}
on the lateral surfaces  of the thin cylinders (the index $i$ indicates the corresponding cylinder), and
\begin{equation}\label{in_2+}
\big(-  \varepsilon \,  \mathbb{D}_\varepsilon(x)\nabla_x u_\varepsilon   
\big) \cdot \boldsymbol{\nu}_\varepsilon   =  \varphi^{(0)}_\varepsilon(x,t)
\end{equation}
on the node surface (identified by the index $0$)
close the problem.  Note that the boundary condition for the node is of pure Neumann type, see also Remark  \ref{rem_node}
in Section \ref{Sec:Statement}.

This  thin three-dimensional junction shrunks into a three-part graph (see Fig.~\ref{f3})
and the elliptic part of the differential equation \eqref{in_1}  disappears as the parameter $\varepsilon$ tends to zero.
Thus, for some  limit longitudinal velocity component denoted by  $v^{(i)}$
one expects  that the limit concentration   $ w^{(i)}$  satisfies the  first-order hyperbolic differential equation
\begin{equation}\label{in_3}
\partial_t w ^{(i)}+ \partial_{x_i} \big(    v^{(i)}(x_i)  w^{(i)}\big) = 0
\end{equation}
on the $i$th edge  of the graph, $i\in \{1,2,3\}$. The limit concentration laws \eqref{in_3}  are coupled  with some gluing condition at the vertex.

 How to find these limit equations, how to find the correct matching conditions at the vertex, how to justify all this, and how to characterize  the effect of the  surface interactions \eqref{in_2}, and especially the condition  \eqref{in_2+} at the node boundary,  for the  wellposedness of the  transport process in the entire thin graph-like junction?  These are the main questions that will be answered in this article.
\newline

 In various fields, first-order (typically nonlinear) hyperbolic equations on networks  like \eqref{in_3}  have  gained interest  and led to a wide range of  literature in recent years (see, e.g., \cite[Chapt.~3]{Bay_Mon_Gar_Goa_Pic_2022},\cite{Gar_Piccoli_2009,Bre_Ngu,Mar_Piccoli,Mus_Fjo_Ris_2022} for traffic models and \cite{Gug_Her_2022} for models of gas flow,  and the references therein).
As was noted in \cite{Gar_Piccoli_2009},  the dynamics at a vertex is not uniquely determined by imposing the conservation of mass through it and
to fully describe the evolution of the process being modeled over the entire graph, the first step is to define the concept of the solution at the vertex. A lot of  different additional conditions at the vertex have been proposed, especially  flux conditions for the  traffic flow models  on networks  as mentioned above. 
Some researchers used the vanishing viscosity method to prove the existence of conservation laws on graphs,  mainly for the Cauchy problem on unbounded graphs. Namely,  they add to each conservation equation  a small viscosity $\nu \, \partial_{x_i x_i} w_i$ and consider the corresponding parabolic Cauchy problem on the graph, which has a unique solution, and then prove that at least for a subsequence $(\nu_k \to 0)$   the solutions of the parabolic problem converge to a solution of the conservation law problem in some sense under special conditions for the flux  (see e.g. \cite{Coc_Gar_Viscosity}).

In this contribution we take a different point of view. The main reason why one-dimensional models on graphs do not fully describe the dynamics  
at a vertex of a graph is the inherently three-dimensional character of  the processes.   A lot of information and features are lost when the three-dimensional node shrinks into the vertex of a graph.

 Our approach allows us to take into account all these features in the limit (as $\varepsilon \to 0)$ from a three-dimensional model to the corresponding one-dimensional model on a graph.
A rigorous procedure for constructing and justifying the complete asymptotic expansion of the solution has been developed, which is the main novelty of this article. Compared to \cite{Mel-Klev_AA-2022}, where an elliptic problem with a predominance of convection in a thin graph-like junction was considered, there are significant generalizations and improvements in this paper. Namely, we have waived special assumptions for the vector field $\overrightarrow{V_\varepsilon};$ the thin cylinders formed a junction can be now curvilinear; inhomogeneous boundary conditions on the node boundary are considered and its influence is studied on the asymptotic behaviour of the solution.
All this became possible thanks to the proof of the maximum principle (Lemma~\ref{the maximum principle})  for solutions of
convection-dominated parabolic problems in thin graph-like junctions.

Studies of various physical and biological processes in thin channels and connections are relevant for many areas of natural science (see, e.g.,  \cite{Pan_2005,Post-2012}). A review of works related to problems of diffusion-convection in thin tubular structures is given in \cite{Mel-Klev_AA-2022}. Here we mention the works \cite{Ben_Paz_2016,Bun_Gau_2022,Mar_2019} that we discovered while working on this article.  The steady Bingham flow was studied in a two-dimensional thin $Y$-like shaped structure and  three uncoupled problems were obtained in the limit in \cite{Bun_Gau_2022}; in the other ones, the incompressible stationary fluids flowing through multiple pipe systems were studied via asymptotic analysis. Another related scenario occurs for   thin  debris-filled wellbores or when   modelling effective surface roughness in any kind of tubes and channels with varying aperture (see e.g. \cite{BurbullaRohde22,BurbullaHoerlRohde22} for such models). Parabolic problems with a predominance of convection in thin graph-shaped junctions, as far as we know, have not been studied at all.\\


After the problem statement in Section ~\ref{Sec:Statement}, the paper is structured as follows.  Section~\ref{Sec:expansions} deals with the construction  of a formal asymptotic expansion for the solution to the problem \eqref{probl}. It consists of three parts: the regular part located inside each thin cylinder (\S~\ref{regular_asymptotic}), the boundary-layer part located near the bases of some thin cylinders (\S~\ref{subsec_Bound_layer}), and the inner part discovered in a neighborhood of
the node (\S~\ref{subsec_Inner_part}).  The terms of the node-layer part of the asymptotics  are special solutions to boundary value problems in an unbounded domain with different outlets at infinity; in fact, they describe the  dynamics in the node. It turns out that such solutions have polynomial growth at infinity. Writing down the solvability conditions for such problems, transmission conditions are derived at the vertex of the graph for terms of the regular expansion. In particular, the limit problem on the graph is obtained. It depends on the dynamics of the field $\overrightarrow{V_\varepsilon}$ (see \S~\ref{sub_limit_problem} and  \S~\ref{sub_limit_problem+}, where the representation for the solution is given as well). In Section~\ref{Sec:justification}, we construct a complete asymptotic expansion  in the whole thin graph-like junction  and  calculate residuals that  its partial sum leaves in the problem \eqref{probl}. The maximum principle for the solution, as well as the $L^2$-estimate for its gradient are proved in Section~\ref{A priori estimates}. Our main result is Theorem \ref{Th_1}  that justifies  the constructed asymptotic expansion  and  provides the main asymptotic estimates. In the final section, the obtained results are discussed. In particular,  we outline the
fundamentally different  limit problems that occur for the two different flow regimes (two inlets, one outlet or  (one inlet,  two outlets) and relate them to the aforementioned studies on pure network models.
Prospects for further research are given.


\section{Problem statement}\label{Sec:Statement}
First,  for some small parameter $\varepsilon >0$, we describe a prototype thin graph-like junction  $\Omega_\varepsilon$.  It  consists of three thin  cylinders
$$
\Omega_\varepsilon^{(i)} =
  \Bigg\{
  x=(x_1, x_2, x_3)\in\Bbb{R}^3 \ : \
  \varepsilon \,  \ell_0  <x_i<\ell_i, \quad
  \sum\limits_{j=1}^3 (1-\delta_{ij})x_j^2<\varepsilon^2 h_i^2(x_i)
  \Bigg\}, \quad i=1,2,3,
$$
that are joined through a domain $\Omega_\varepsilon^{(0)}$ (referred to  as the "node"). Let   $\ell_0\in(0, \frac13), \ \ell_i\geq1, \   i\in\{1,2,3\}$ be given.
 The aperture function $h_i=h_i(x_i), \ x_i \in [0, \ell_i],$ is positive and  belongs to the space $C^3 ([0, \ell_i]).$ We assume that it is  equal to some constants in neighborhoods of the points $x=0$ and $x_i= \ell_i$ $(i\in\{1,2,3\})$.
The symbol $\delta_{ij}$ denotes  the Kronecker delta, i.e.,
$\delta_{ii} = 1$ and $\delta_{ij} = 0$ if $i \neq j.$

We denote the lateral surface of the thin cylinder $\Omega_\varepsilon^{(i)}$ by
$$
{\Gamma_\varepsilon^{(i)}} := \partial\Omega_\varepsilon^{(i)} \cap \{ x\in\Bbb{R}^3 \ : \ \varepsilon \ell_0<x_i<\ell_i \}
$$
and by
$$
\Upsilon_\varepsilon^{(i)} (y_i) := \Omega_\varepsilon^{(i)} \cap
\big\{  x\in\Bbb{R}^3 \, : \ x_i= y_i\big\}
$$
its cross-section at the point $y_i \in [ \varepsilon \ell_0, \ell_i].$
\begin{figure}[htbp]
\begin{center}
\includegraphics[width=5cm]{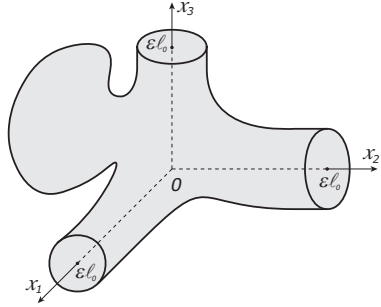}
\end{center}
\caption{The node $\Omega_\varepsilon^{(0)}$}\label{f2}
\end{figure}

The node $\Omega_\varepsilon^{(0)}$ (see Fig.~\ref{f2}) is formed by the homothetic transformation with coefficient $\varepsilon$ from a bounded domain $\Xi^{(0)}\subset \Bbb R^3$ containing the origin,  i.e.,
$
\Omega_\varepsilon^{(0)} = \varepsilon\, \Xi^{(0)}.
$
In addition, we assume that the boundary $\partial  \Xi^{(0)}$ of $\Xi^{(0)}$ contains the disks
$\Upsilon^{(i)}_1(\ell_0) := \overline{\Xi^{(0)}} \cap
\big\{ x\colon x_i= \ell_0\big\}, \   i\in\{1,2,3\},$ that are the bases of some right cylinders, respectively, and the lateral surfaces of these cylinders belong to $\partial  \Xi^{(0)}.$ Thus, the boundary of the node $\Omega_\varepsilon^{(0)}$ consists of
the disks $\Upsilon_\varepsilon^{(i)} (\varepsilon\ell_0), \   i\in\{1,2,3\},$
and the surface
$$
\Gamma_\varepsilon^{(0)} :=
\partial\Omega_\varepsilon^{(0)} \backslash
\left\{
 \overline{\Upsilon_\varepsilon^{(1)} (\varepsilon \ell_0)} \cup
 \overline{\Upsilon_\varepsilon^{(2)} (\varepsilon \ell_0)} \cup
 \overline{\Upsilon_\varepsilon^{(3)} (\varepsilon \ell_0)}
\right\}.
$$

Hence,  the model thin graph-like junction  $\Omega_\varepsilon$  (see Fig.~\ref{f3})
is   the interior of the union
$
\bigcup_{i=0}^{3}\overline{\Omega_\varepsilon^{(i)}}
$
and
the surface $\partial\Omega_\varepsilon\setminus \bigcup_{i=0}^{3}\overline{\Upsilon_\varepsilon^{(i)} (\ell_i)}$ is smooth of the class $C^3.$

\begin{figure}[htbp]
\begin{center}
\includegraphics[width=8cm]{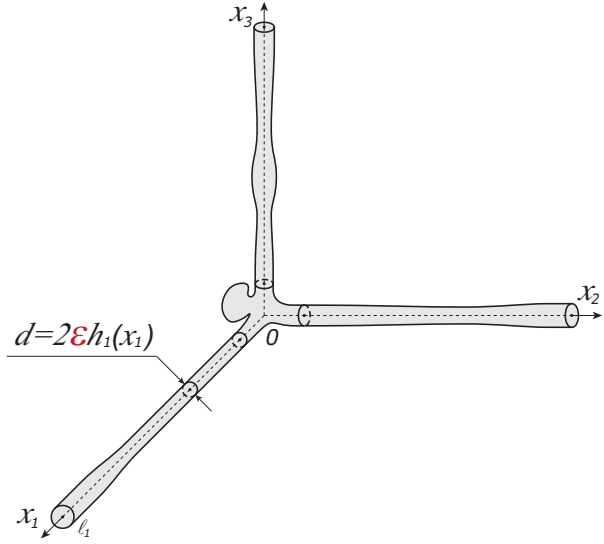}
\end{center}
\caption{The model thin graph-like  junction $\Omega_\varepsilon$}\label{f3}
\end{figure}

In $\Omega_\varepsilon,$ we consider the following parabolic convection-diffusion  problem:
\begin{equation}\label{probl}
\left\{\begin{array}{rcll}
 \partial_t u_\varepsilon -  \varepsilon\, \mathrm{div}_x \big( \mathbb{D}_\varepsilon(x) \nabla_x u_\varepsilon\big) +
  \mathrm{div}_x \big( \overrightarrow{V_\varepsilon}(x) \, u_\varepsilon\big)
  & = & 0, &
    \text{in} \ \Omega_\varepsilon \times (0, T),
\\[2mm]
\big(-  \varepsilon \,  \mathbb{D}_\varepsilon\nabla_x u_\varepsilon +  u_\varepsilon \, \overrightarrow{V_\varepsilon}\big) \cdot \boldsymbol{\nu}_\varepsilon &  = & \varepsilon \varphi^{(i)}_\varepsilon &
   \text{on} \ \Gamma^{(i)}_\varepsilon \times (0, T), \ \ i \in \{1, 2, 3\},
\\[2mm]
\big(-  \varepsilon \,  \mathbb{D}_\varepsilon\nabla_x u_\varepsilon +  u_\varepsilon \, \overrightarrow{V_\varepsilon}\big) \cdot \boldsymbol{\nu}_\varepsilon &  = &  \varphi^{(0)}_\varepsilon &
   \text{on} \ \Gamma^{(0)}_\varepsilon \times (0, T),
\\[2mm]
 u_\varepsilon \big|_{x_i= \ell_i}
 & = & q_i, & \text{on} \ \Upsilon_{\varepsilon}^{(i)} (\ell_i)\times (0, T),
\\[2mm]
u_\varepsilon\big|_{t=0}&=& 0, & \text{on} \ \Omega_{\varepsilon},
 \end{array}\right.
\end{equation}
where ${\boldsymbol{\nu}}_\varepsilon$ is the outward unit normal to $\partial \Omega_\varepsilon.$

{\bf A1. \ Assumptions for the diffusion matrix.}  It has the following structure:
$$
\mathbb{D}_\varepsilon(x) =
\left(
\begin{matrix}
  a^{(0)}_{11}(\frac{x}{\varepsilon}) & a^{(0)}_{12}(\frac{x}{\varepsilon}) & a^{(0)}_{13}(\frac{x}{\varepsilon}) \\[2mm]
  a^{(0)}_{21}(\frac{x}{\varepsilon}) & a^{(0)}_{22}(\frac{x}{\varepsilon}) & a^{(0)}_{23}(\frac{x}{\varepsilon}) \\[2mm]
  a^{(0)}_{31}(\frac{x}{\varepsilon}) & a^{(0)}_{32}(\frac{x}{\varepsilon}) & a^{(0)}_{33}(\frac{x}{\varepsilon})
\end{matrix}
\right)  =: \mathbb{D}_\varepsilon^{(0)}(x), \qquad x \in \ \Omega_\varepsilon^{(0)},
$$
$$
\mathbb{D}_\varepsilon(x)  =
\left(
\begin{matrix}
  a^{(1)}_{11} & 0 & 0 \\[2mm]
  0 & a^{(1)}_{22}(\frac{\overline{x}_1}{\varepsilon}) &  a^{(1)}_{23}(\frac{\overline{x}_1}{\varepsilon}) \\[2mm]
  0 &  a^{(1)}_{32}(\frac{\overline{x}_1}{\varepsilon}) &  a^{(1)}_{33}(\frac{\overline{x}_1}{\varepsilon})
\end{matrix}
\right)
 =:  \mathbb{D}_\varepsilon^{(1)}(x), \qquad x \in \ \Omega_\varepsilon^{(1)},
$$
$$
\mathbb{D}_\varepsilon(x)  =
\left(
\begin{matrix}
  a^{(2)}_{11}(\frac{\overline{x}_2}{\varepsilon}) & 0 & a^{(2)}_{13}(\frac{\overline{x}_2}{\varepsilon}) \\[2mm]
  0 & a^{(2)}_{22} & 0 \\[2mm]
  a^{(2)}_{31}(\frac{\overline{x}_2}{\varepsilon}) & 0 & a^{(2)}_{33}(\frac{\overline{x}_2}{\varepsilon})
\end{matrix}
\right)
 =:  \mathbb{D}_\varepsilon^{(2)}(x), \qquad x \in \ \Omega_\varepsilon^{(2)},
$$
$$
\mathbb{D}_\varepsilon(x) =
\left(
\begin{matrix}
  a^{(3)}_{11}(\frac{\overline{x}_3}{\varepsilon}) & a^{(3)}_{12}(\frac{\overline{x}_3}{\varepsilon}) & 0 \\[2mm]
  a^{(3)}_{21}(\frac{\overline{x}_3}{\varepsilon}) & a^{(3)}_{22}(\frac{\overline{x}_3}{\varepsilon}) & 0 \\[2mm]
  0 & 0 & a^{(3)}_{33}
\end{matrix}
\right)
  =:  \mathbb{D}_\varepsilon^{(3)}(x), \qquad x \in \ \Omega_\varepsilon^{(3)},
$$
where
$$
\overline{x}_i =
\left\{\begin{array}{lr}
(x_2, x_3), &\text{if} \ \ i=1, \\
(x_1, x_3), &\text{if} \ \ i=2, \\
(x_1, x_2), & \text{if} \ \ i=3.
\end{array}\right.
$$
This dependence of the matrix components on the variables $x$ and the parameter $\varepsilon$ makes it possible to take into account the heterogeneous physical structure of both thin cylinders and the node.

In addition, without loss of generality, we assume that  the components of the diffusion matrix $\mathbb{D}_1(x)$ belong to the
H\!$\mathrm{\ddot{o}}$\!lder space $C^{1,\alpha}\big(\overline{\mho}_1\big),$ where   $\mho_1$ is a domain that includes the junction $\Omega_\varepsilon$
for all $\varepsilon \in (0, 1).$  This assumption means that
$$
\mathbb{D}_\varepsilon^{(0)}(x)\big|_{x_i = \varepsilon \, \ell_0} = \mathbb{D}_\varepsilon^{(i)}(x)\big|_{x_i = \varepsilon \, \ell_0}, \quad i\in\{1, 2, 3\}.
$$

All matrices $\{\mathbb{D}_\varepsilon^{(i)}\}_{i=0}^3$ are symmetric, i.e., $a^{(i)}_{mn} = a^{(i)}_{nm},$ and   there exist positive constants $\kappa_0^{(i)}$ and $ \kappa_1^{(i)}$  such that  for all $\varepsilon \in (0, 1),$ $x \in \Omega_\varepsilon,$
and $y \in \Bbb R^3$
\begin{equation}\label{n1}
  \kappa_0^{(i)} \, |y|^2 \le \sum_{m,n=1}^{3} a^{(i)}_{mn}(x,\varepsilon) y_m y_n \le \kappa_1^{(i)} \, |y|^2, \quad i\in\{0, 1, 2, 3\}.
\end{equation}
From  \eqref{n1} it follows that the constants   $\{a^{(i)}_{ii}\}_{i=1}^3$ are positive.

\smallskip

{\bf A2. \ Assumptions for the vector field.}
The vector-valued function $\overrightarrow{V_\varepsilon}$  depends also on the parts of the thin junction $\Omega_\varepsilon,$ namely,
$$
\overrightarrow{V_\varepsilon}(x)=
\left(v^{(0)}_1(\tfrac{x}{\varepsilon}),  \  v^{(0)}_2(\tfrac{x}{\varepsilon}), \  v^{(0)}_3(\tfrac{x}{\varepsilon})
\right) =: \overrightarrow{V_\varepsilon}^{(0)}(x), \qquad x \in \ \Omega_\varepsilon^{(0)},
$$
$$
\overrightarrow{V_\varepsilon}(x)=
\left(v^{(1)}_1(x_1), \ \varepsilon\,  v^{(1)}_2(x_1, \tfrac{\overline{x}_1}{\varepsilon}), \
 \varepsilon \, v^{(1)}_3(x_1,\tfrac{\overline{x}_1}{\varepsilon})
\right) =: \overrightarrow{V_\varepsilon}^{(1)}(x), \qquad x \in \ \Omega_\varepsilon^{(1)} \quad (v^{(1)}_1 < 0),
$$
$$
\overrightarrow{V_\varepsilon}(x) =
\left( \varepsilon\,  v^{(2)}_1(x_2, \tfrac{\overline{x}_2}{\varepsilon}), \   v^{(2)}_2(x_2),  \
 \varepsilon \, v^{(2)}_3(x_2,\tfrac{\overline{x}_2}{\varepsilon})
\right) =: \overrightarrow{V_\varepsilon}^{(2)}(x), \qquad x \in \ \Omega_\varepsilon^{(2)} \quad (v^{(2)}_2 <  0 \ \  \text{or} \ \ v^{(2)}_2 > 0),
$$
$$
\overrightarrow{V_\varepsilon}(x) =
\left(\varepsilon \,  v^{(3)}_1(x_3, \tfrac{\overline{x}_3}{\varepsilon}), \
    \varepsilon \, v^{(3)}_2(x_3,\tfrac{\overline{x}_3}{\varepsilon}), \ v^{(3)}_3(x_3)
\right) =: \overrightarrow{V_\varepsilon}^{(3)}(x), \qquad x \in \ \Omega_\varepsilon^{(3)} \quad (v^{(3)}_3 >  0).
$$
We see that the main directions (in magnitude) of the vector-valued function $\overrightarrow{V_\varepsilon}^{(i)}$ is routed along the axes of the thin cylinder $\Omega_\varepsilon^{(i)}.$

All components of the vector-valued function $\overrightarrow{V_1}$ belong  also to the space $C^{1,\alpha}\big(\overline{\mho}_1\big).$
In addition, for each $i\in\{1,2,3\}$ the function $v^{(i)}_i(x_i),$ defined on the segment $[0, \ell_i],$ is equal to a constant $\mathrm{v}_i$ in
a neighbourhood  of the origin and the other components of $\overrightarrow{V_\varepsilon}^{(i)}$
have compact supports with respect to the corresponding longitudinal  variable $x_i$ in $\Omega^{(i)}_\varepsilon.$
This means that
\begin{gather}
\overrightarrow{V_\varepsilon}^{(0)}\big|_{x_1=\varepsilon \, \ell_0-0} = \overrightarrow{V_\varepsilon}^{(1)}\big|_{x_1=\varepsilon \, \ell_0+0} = \left(
\mathrm{v}_1, \, 0,\, 0\right) \ \  \text{on} \ \Upsilon_\varepsilon^{(1)} (\varepsilon\ell_0), \notag
\\
\overrightarrow{V_\varepsilon}^{(0)}\big|_{x_2=\varepsilon \, \ell_0-0} = \overrightarrow{V_\varepsilon}^{(2)}\big|_{x_2=\varepsilon \, \ell_0+0} = \left(0,\, \mathrm{v}_2, \, 0\right) \ \  \text{on} \ \Upsilon_\varepsilon^{(2)} (\varepsilon\ell_0), \notag
 \\ \label{v_i}
\overrightarrow{V_\varepsilon}^{(0)}\big|_{x_3=\varepsilon \, \ell_0-0} = \overrightarrow{V_\varepsilon}^{(3)}\big|_{x_3=\varepsilon \, \ell_0+0} =\left(
0,\, 0, \, \mathrm{v}_3\right) \ \  \text{on} \ \Upsilon_\varepsilon^{(3)} (\varepsilon\ell_0).
\end{gather}

Depending on the sign of the function $v^{(2)}_2,$ we have two movements (dynamics) for  the velocity field $\overrightarrow{V_\varepsilon},$ namely,
\begin{enumerate}
  \item it   enters  the cylinders
$\Omega_\varepsilon^{(1)}$ and $\Omega_\varepsilon^{(2)}$ and outgoes the  cylinder $\Omega_\varepsilon^{(3)}$ \quad $(v^{(1)}_1 < 0, \ v^{(2)}_2 <  0, \ v^{(3)}_3 >  0);$

\smallskip

  \item it  enters  the cylinder $\Omega_\varepsilon^{(1)}$  and outgoes the cylinders $\Omega_\varepsilon^{(2)}$ and $\Omega_\varepsilon^{(3)}$ \quad $(v^{(1)}_1 < 0, \ v^{(2)}_2 >  0, \ v^{(3)}_3 >  0).$
\end{enumerate}

To describe the dynamics of the velocity field $\overrightarrow{V_\varepsilon}$ in the node $\Omega_\varepsilon^{(0)},$ we make the following assumptions. The velocity field is conservative  in the node and
its potential $p$ is a solution to the boundary value problem
\begin{equation}\label{potential0}
\left\{\begin{array}{rcll}
    \Delta_\xi p(\xi)  & = & 0, & \quad
    \xi \in \Xi^{(0)},
\\[2mm]
\dfrac{\partial p(\xi)}{\partial \xi_i}  &=&  \mathrm{v}_i, & \quad
   \xi \in \Upsilon^{(i)}_1(\ell_0), \quad i\in \{1, 2, 3\},
\\[4mm]
\dfrac{\partial p(\xi)}{\partial \boldsymbol{\nu}}  &=&  0, & \quad
   \xi \in \Gamma^{(0)} := \partial{\Xi^{(0)}} \setminus \Big( \bigcup_{i=1}^3 \Upsilon^{(i)}_1(\ell_0)\Big),
\end{array}\right.
\end{equation}
where $\Delta_\xi$ is  the Laplace operator,  $\frac{\partial p}{\partial \boldsymbol{\nu}}$ is the derivative along the  outward unit normal $\boldsymbol{\nu}$ to $\partial \Xi^{(0)}.$ The Neumann problem \eqref{potential0} has a  solution if and only if
the  conservation condition
\begin{equation}\label{cond_1}
  \sum_{i=1}^{3} h_i^2(0) \, \mathrm{v}_i = 0
\end{equation}
holds.  To ensure   unique solvability of \eqref{potential0}  we  enforce the condition $\int_{\Xi^{(0)}} p(\xi)\, d\xi =0.$
Thus,
\begin{equation}\label{field_0}
  \overrightarrow{V_\varepsilon}^{(0)}(x)  = \nabla_{\xi} p(\xi)\big|_{\xi = \frac{x}{\varepsilon}} = \varepsilon \nabla_x \big( p(\tfrac{x}{\varepsilon}) \big), \quad x\in \Omega^{(0)}_\varepsilon,
\end{equation}
and, clearly,  $\mathrm{div}_x \overrightarrow{V_\varepsilon}^{(0)} = 0$ in $\Omega^{(0)}_\varepsilon,$ i.e.,
 $\overrightarrow{V_\varepsilon}^{(0)}$  is incompressible in $\Omega^{(0)}_\varepsilon.$
\begin{remark}\label{rem_node}
From the assumption about the conservativeness of the vector field $\overrightarrow{V_\varepsilon}^{(0)}$  it follows  that there is no flow of $\overrightarrow{V_\varepsilon}^{(0)}$  through the surface $\Gamma^{(0)}_\varepsilon$ and the amount of the flow entering the node $\Omega^{(0)}_\varepsilon$ is equal to the amount of the flow outgoing it. Thus, the condition on the node boundary reduces to
$$
\big(-  \varepsilon \,  \mathbb{D}_\varepsilon\nabla_x u_\varepsilon \big) \cdot \boldsymbol{\nu}_\varepsilon   =   \varphi^{(0)}_\varepsilon .
$$
\end{remark}

\smallskip

{\bf A3. \ Assumptions for the  boundary data.} They are determined as follows:
$$
\varphi^{(0)}_\varepsilon(x,t) := \varphi^{(0)}\big(\tfrac{x}{\varepsilon},t\big), \quad (x,t) \in \overline{\Gamma_\varepsilon^{(0)}} \times [0,T],
$$
$$
 \varphi^{(i)}_\varepsilon(x,t) := \varphi^{(i)}\big(\tfrac{\overline{x}_i}{\varepsilon}, x_i, t\big), \quad (x,t) \in \overline{\Gamma_\varepsilon^{(i)}} \times [0,T], \quad i\in \{1,2,3\},
$$
where  the function $\varphi^{(0)}(\xi,t), \ (\xi,t) \in \overline{\Xi^{(0)}}\times [0,T],$  and the functions
$$
\varphi^{(i)}(\overline{\xi}_i,x_i,t), \ \ (\overline{\xi}_i,x_i,t) \in \big\{ |\overline{\xi}_i|\le h_i(x_i), \ x_i\in [0, \ell_i], \ t\in [0,T]\big\}, \quad  i\in\{1,2,3\},
$$
belong to the class $C^2$ with respect to the corresponding  variables $\xi$ and $x$ and to $C^1$ with respect to $t$ in their domains of definition.
In addition  $\varphi^{(0)}$ vanishes uniformly with respect to $t\in [0, T]$ in neighborhoods of
$\Upsilon^{(i)}_1(\ell_0), \ i\in \{1,2,3\},$
the functions $\{\varphi^{(i)}\}_{i=1}^3$ vanish  uniformly with respect to $t$  and $\overline{\xi}_i$ in neighborhoods of the ends  of the corresponding closed interval  $[0, \ell_i].$

The functions $\{q_i(t), \ t\in [0, T]\}_{i=1}^3 $ are smooth and nonnegative.
Moreover, the matching conditions
\begin{equation}\label{match_conditions}
  q_i(0) = q'_i(0)=0, \quad i\in \{1,2,3\}, \quad \text{and} \quad \varphi^{(i)}\big|_{t=0}=0, \quad i\in \{0, 1,2,3\},
\end{equation}
are satisfied.

\medskip

Thus, due to the classical theory of parabolic initial-boundary problem there exists  a unique classical solution to the problem \eqref{probl}  (see e.g. \cite[Chapt. IV, \S 5]{Lad_Sol_Ura_1968}).

\medskip

Our goal is to study the asymptotic behavior of the solution to the problem \eqref{probl} as
$\varepsilon \to 0,$ i.e., when the thin junction $\Omega_\varepsilon$ is shrunk into the graph
$$
 \mathcal{I} := I_1 \cup I_2 \cup I_3,
$$
 where $I_i := \{x\colon  x_i \in  [0, \ell_i], \ \overline{x}_i  = (0, 0)\},$ namely,
 \begin{itemize}
   \item
   develop a procedure for constructing a complete asymptotic expansion of the solution $u_\varepsilon$  as $\varepsilon \to 0,$
   \item
   justify this procedure and prove the corresponding asymptotic estimates,
   \item
   derive the corresponding limit problem $(\varepsilon = 0)$ and prove the existence of its solution,
   \item
 prove energy and uniform pointwise estimates for the difference between the solution to
the problem~\eqref{probl} and the solution to the limit problem.
 \end{itemize}
\begin{remark}
The assumptions made above, except for the conservation property of the vector field $\overrightarrow{V_\varepsilon}$ in the node $\Omega_\varepsilon^{(0)},$ are necessary for the existence of a classical solution to the problem \eqref{probl}. This property of the vector field will be used to prove the maximum principle for the solutions of parabolic problems in thin graph-like junctions  and to estimate the $L^2$-norm of their gradients (see Lemma \ref{the maximum principle}  and Corollary~\ref{apriory_estimate}).
 To construct the asymptotic expansion for the solution, we will need additional assumptions about the smoothness of the coefficients and given functions,  which  are described in Remark~\ref{differ}.
\end{remark}

\section{Formal asymptotic expansions}\label{Sec:expansions}
\subsection{Regular part of the asymptotics}\label{regular_asymptotic}
We seek it  in the form
\begin{equation}\label{regul}
\mathfrak{U}_\varepsilon^{(i)} := w_0^{(i)} (x_i,t) + \sum\limits_{k=1}^{+\infty} \varepsilon^{k}
    \left(w_k^{(i)} (x_i,t) + u_k^{(i)} \Big( x_i, \dfrac{\overline{x}_i}{\varepsilon}, t \Big)
     \right).
\end{equation}
For each $i \in \{1,2,3\},$ the ansatz $\mathfrak{U}_\varepsilon^{(i)}$ is located inside of the thin cylinder $\Omega^{(i)}_\varepsilon$ and their  coefficients depend both on the variable $t,$ the corresponding longitudinal variable $x_i$ and so-called "fast variables" $\frac{\overline{x}_i}{\varepsilon}.$

Formally substituting $\mathfrak{U}_\varepsilon^{(i)}$  into the corresponding differential equation of the problem~\eqref{probl} and collecting terms of the same powers of $\varepsilon,$ we obtain
$$
\dot{w}_0^{(i)} (x_i,t) + \sum\limits_{k=1}^{+\infty} \varepsilon^{k}
    \left(\dot{w}_k^{(i)} (x_i,t) + \dot{u}_k^{(i)} \left( x_i, \bar{\xi}_i, t \right)
     \right)
$$
$$
        - \mathrm{div}_{\bar{\xi}_i}
    \big(
        \tilde{\mathbb{D}}^{(i)}_{\bar{\xi}_i}
        \nabla_{\bar{\xi}_i} u^{(i)}_1(x_i, \bar{\xi}_i, t)
    \big)
    +
    \big(  v_i^{(i)}(x_i) \, w^{(i)}_0(x_i, t) \big)^\prime
+  w^{(i)}_0(x_i,t) \,   \mathrm{div}_{\bar{\xi}_i} \big( \overline{V}^{(i)}(x_i, \bar{\xi}_i) \big)
$$
$$
    + \sum\limits_{k=1}^{+\infty}\varepsilon^{k}
    \Bigg(
        -
        \mathrm{div}_{\bar{\xi}_i}
        \big(
            \tilde{\mathbb{D}}^{(i)}_{\bar{\xi}_i}
            \nabla_{\bar{\xi}_i} u^{(i)}_{k+1}(x_i, \bar{\xi}_i,t)
        \big)
        -
        a_{ii}^{(i)}
        \Big(
            w^{(i)}_{k-1}(x_i,t)
            +
            u^{(i)}_{k-1}(x_i, \bar{\xi}_i,t)
        \Big)^{\prime\prime}
$$
\begin{equation}\label{rel_1}
        +  \Big(  v^{(i)}(x_i) \, \big[ w^{(i)}_k(x_i,t) + u^{(i)}_{k}(x_i, \bar{\xi}_i,t) \big]
        \Big)^\prime
        +
        \mathrm{div}_{\bar{\xi}_i}
        \big(
            \overline{V}^{(i)}(x_i, \bar{\xi}_i) \,
            \big[ w^{(i)}_{k}(x_i,t) + u^{(i)}_{k}(x_i, \bar{\xi}_i,t) \big]
        \big)
    \Bigg) \approx 0,
\end{equation}
where the dot over a function denotes the derivative with respect to $t,$  the symbol ``$\prime$'' denotes the derivative with respect to the longitudinal variable $x_i,$
$\bar{\xi}_i=\frac{\bar{x}_i}{\varepsilon},$
$u^{(i)}_{0} \equiv 0,$
\begin{gather}\notag
\overline{V}^{(1)}(x_1, \bar{\xi}_1) =
\Big( v^{(1)}_2(x_1, \bar{\xi}_1), \, v^{(1)}_3(x_1,\bar{\xi}_1) \Big),
\qquad
\overline{V}^{(2)}(x_2, \bar{\xi}_2) =
\Big( v^{(2)}_1(x_2, \bar{\xi}_2), \, v^{(2)}_3(x_2,\bar{\xi}_2) \Big),
\\[2mm] \label{V-2D}
\overline{V}^{(3)}(x_3, \bar{\xi}_3) =
\Big( v^{(3)}_1(x_3, \bar{\xi}_3), \, v^{(3)}_2(x_3,\bar{\xi}_3) \Big),
\end{gather}
and
\begin{equation}\label{mat-2D}
  \tilde{\mathbb{D}}^{(1)}_{\bar{\xi}_1}
 =
\left(
    \begin{matrix}
        a^{(1)}_{22}(\bar{\xi}_1) & a^{(1)}_{23}(\bar{\xi}_1)
        \\[2mm]
        a^{(1)}_{32}(\bar{\xi}_1) & a^{(1)}_{33}(\bar{\xi}_1)
    \end{matrix}
\right),
\qquad
\tilde{\mathbb{D}}^{(2)}_{\bar{\xi}_2}
 =
\left(
    \begin{matrix}
        a^{(2)}_{11}(\bar{\xi}_2) & a^{(2)}_{13}(\bar{\xi}_2)
        \\[2mm]
        a^{(2)}_{31}(\bar{\xi}_2) & a^{(2)}_{33}(\bar{\xi}_2)
    \end{matrix}
\right),
\qquad
\tilde{\mathbb{D}}^{(3)}_{\bar{\xi}_3}
 =
\left(
    \begin{matrix}
        a^{(3)}_{11}(\bar{\xi}_3) & a^{(3)}_{12}(\bar{\xi}_3)
        \\[2mm]
        a^{(3)}_{21}(\bar{\xi}_3) & a^{(3)}_{22}(\bar{\xi}_3)
    \end{matrix}
\right).
\end{equation}

Equating the coefficients of the same powers of $\varepsilon$ to zero in~\eqref{rel_1}, we get the following differential equations:
\begin{equation}\label{eq_1}
  \mathrm{div}_{\bar{\xi}_i}
    \big(
        \tilde{\mathbb{D}}^{(i)}_{\bar{\xi}_i}
        \nabla_{\bar{\xi}_i} u^{(i)}_1(x_i, \bar{\xi}_i, t)
    \big)
    =
    \big(  v_i^{(i)}(x_i) \, w^{(i)}_0(x_i, t) \big)^\prime +  w^{(i)}_0(x_i,t) \,   \mathrm{div}_{\bar{\xi}_i} \big( \overline{V}^{(i)}(x_i, \bar{\xi}_i) \big)+ \dot{w}_0^{(i)} (x_i,t), \quad \overline{\xi}_i\in\Upsilon_i (x_i),
\end{equation}
where \ $\Upsilon_i(x_i) := \big\{ \overline{\xi}_i\in\Bbb{R}^2 \colon |\overline{\xi}_i|< h_i(x_i) \big\};$
and for $k \in \Bbb N$
\begin{align}\label{eq_3}
  \mathrm{div}_{\bar{\xi}_i}
        \big(
            \tilde{\mathbb{D}}^{(i)}_{\bar{\xi}_i}
            \nabla_{\bar{\xi}_i} u^{(i)}_{k+1}(x_i, \bar{\xi}_i, t)
        \big) = &  \ \Big(  v^{(i)}(x_i) \, \big[ w^{(i)}_k(x_i,t) + u^{(i)}_{k}(x_i, \bar{\xi}_i,t) \big]
        \Big)^\prime  + \dot{w}_k^{(i)} (x_i,t) +   \dot{u}_k^{(i)} \left( x_i, \bar{\xi}_i, t \right) \notag
  \\
  +& \   \mathrm{div}_{\bar{\xi}_i}
        \Big(
            \overline{V}^{(i)}(x_i, \bar{\xi}_i) \,
            \big[ w^{(i)}_{k}(x_i,t) + u^{(i)}_{k}(x_i, \bar{\xi}_i,t) \big]
        \Big)\notag
            \\
             - & \
        a_{ii}^{(i)}
        \Big( w^{(i)}_{k-1}(x_i,t)
            +
            u^{(i)}_{k-1}(x_i, \bar{\xi}_i,t)
        \Big)^{\prime\prime},
\qquad  \bar{\xi}_i \in \Upsilon_i (x_i).
\end{align}

The outward unit normal  to the lateral surface of the thin cylinder $\Omega^{(i)}_\varepsilon$
is as follows
\begin{equation}\label{normal_1}
  {\boldsymbol{\nu}}_\varepsilon = \frac{1}{\sqrt{1 + \varepsilon^2 |h'_i(x_i)|^2}} \Big( - \varepsilon h'_i(x_i), \bar{\nu}_{\bar{\xi}_i}\Big),
\end{equation}
where $\bar{\nu}_{\bar{\xi}_i}$ is the outward unit normal to the boundary of the
disk  $\Upsilon_i(x_i).$
In \eqref{normal_1} the term  ``$- \varepsilon h'_i(x_i)$''  is  the $i$th component of the vector ${\boldsymbol{\nu}}_\varepsilon$  $(i\in \{1, 2, 3\}).$

Substituting \eqref{regul} in the boundary condition on the  lateral surface of the thin cylinder $\Omega^{(i)}_\varepsilon$, we deduce the following asymptotic equation:
$$
  \varepsilon a^{(i)}_{ii}\, h'_i(x_i)
  \bigg( \big(w_0^{(i)} (x_i,t)\big)'  + \sum\limits_{k=1}^{+\infty} \varepsilon^{k}
    \left(\big(w_k^{(i)} (x_i,t)\big)' + \big(u_k^{(i)} \left( x_i, \bar{\xi}_i, t \right)\big)'\right)
  $$
  $$
  - \sum\limits_{k=1}^{+\infty} \varepsilon^{k-1} \big(
            \tilde{\mathbb{D}}^{(i)}_{\bar{\xi}_i}
            \nabla_{\bar{\xi}_i} u^{(i)}_{k}(x_i, \bar{\xi}_i, t)\big) \cdot \bar{\nu}_{\bar{\xi}_i}
          -  v^{(i)}_{i}(x_i) \, h'_i(x_i)
  \bigg( w_0^{(i)} (x_i,t)  + \sum\limits_{k=1}^{+\infty} \varepsilon^{k}
    \left(w_k^{(i)} (x_i,t) + u_k^{(i)} \left( x_i, \bar{\xi}_i, t \right)\right) \bigg)
  $$
  $$
    +  \big(\overline{V}^{(i)}(x_i, \bar{\xi}_i) \cdot \bar{\nu}_{\bar{\xi}_i}\big)
    \bigg( w_0^{(i)} (x_i,t)  + \sum\limits_{k=1}^{+\infty} \varepsilon^{k}
    \left(w_k^{(i)} (x_i,t) + u_k^{(i)} \left( x_i, \bar{\xi}_i, t \right)\right)\bigg)
  $$
\begin{equation}\label{asym_relation_1}
     \approx \sqrt{1 + \varepsilon^2 |h'_i(x_i)|^2}\, \varphi^{(i)}(\bar{\xi}_i, x_i,t)
       = \varphi^{(i)}(\bar{\xi}_i, x_i,t) \sum_{k=0}^{+\infty} \frac{(-1)^{k+1} (2k)! |h'_i(x_i)|^{2k}}{4^k (k!)^2 (2k-1)} \, \varepsilon^{2k}
\end{equation}
Equating the coefficients of the same powers of $\varepsilon$ in \eqref{asym_relation_1}, we get
\begin{equation}\label{bc_1}
  \Big(-\tilde{\mathbb{D}}^{(i)}_{\bar{\xi}_i}
            \nabla_{\bar{\xi}_i} u^{(i)}_{1}(x_i, \bar{\xi}_i, t)   +  w_0^{(i)} (x_i,t) \, \overline{V}^{(i)}(x_i, \bar{\xi}_i)\Big) \cdot \bar{\nu}_{\bar{\xi}_i}
                 =   v^{(i)}_{i}(x_i) \, h'_i(x_i)
  \, w_0^{(i)} (x_i,t) + \varphi^{(i)}(\bar{\xi}_i, x_i,t), \quad \bar{\xi}_i \in \partial \Upsilon_i(x_i);
\end{equation}
\begin{align}\label{bc_3}
  \Big(-\tilde{\mathbb{D}}^{(i)}_{\bar{\xi}_i} \nabla_{\bar{\xi}_i} u^{(i)}_{k+1} +
  \big(w_k^{(i)} + u_k^{(i)}\big)\, \overline{V}^{(i)}(x_i, \bar{\xi}_i) \Big) \cdot \bar{\nu}_{\bar{\xi}_i} = &\   v^{(i)}_{i}(x_i) \, h'_i(x_i)
            \left(w_k^{(i)} (x_i,t) + u_k^{(i)} \left( x_i, \bar{\xi}_i, t \right)\right) \notag
            \\
            - & \ a^{(i)}_{ii}\, h'_i(x_i) \left(w_{k-1}^{(i)} (x_i,t) + u_{k-1}^{(i)} \left( x_i, \bar{\xi}_i, t \right)\right)' \notag
            \\
            + &\ \eta^{(i)}_k(x_i)\,  \varphi^{(i)}(\bar{\xi}_i, x_i,t),\qquad \bar{\xi}_i \in \partial \Upsilon_i(x_i)  \ \ (k \in \Bbb N),
\end{align}
where
\begin{equation}\label{eta_k}
  \eta^{(i)}_k(x_i) :=
\left\{
  \begin{array}{ll}
    0, & \hbox{if} \ k \ \hbox{is odd}; \\
    \dfrac{(-1)^{\frac{k}{2}+1} \, k! \, |h'_i(x_i)|^{k}}{4^{\frac{k}{2}} \, ((\frac{k}{2})!)^2 \, (k-1)}, & \hbox{if} \ k \ \hbox{is even}.
  \end{array}
\right.
\end{equation}

Equations \eqref{eq_1} and \eqref{bc_1},  and equations \eqref{eq_3} and \eqref{bc_3}
are Neumann problems in the disk $\Upsilon_i(x_i)$ with respect to the variables $\bar{\xi}_i$ for the fixed index $i\in \{1, 2, 3\},$ and the variables $x_i$ and $t$ are regarded as parameters from $I_\varepsilon^{(i)} \times (0, T)$, where
$$
I_\varepsilon^{(i)} := \{x\colon x_i \in (\varepsilon \ell_0, \ell_i), \quad \bar{x}_i = (0, 0) \}.
$$
For the uniqueness, we supply each of these problems with the corresponding condition
\begin{equation}\label{uniq_1}
  \langle u_k^{(i)}(x_i, \cdot ,  t ) \rangle_{\Upsilon_i(x_i)} :=  \int_{\Upsilon_i(x_i)} u_k^{(i)}(x_i, \bar{\xi}_i,  t ) \, d\bar{\xi}_i = 0 \quad (k\in \Bbb N).
\end{equation}

Writing down the necessary and sufficient condition for the solvability of the problem  \eqref{eq_1} and \eqref{bc_1}, we deduce the following differential equation for the coefficient $w^{(i)}_0$:
\begin{equation}\label{lim_0}
  h^2_i(x_i)\, \partial_t{w}^{(i)}_0(x_i,t) + \partial_{x_i}\big( v_i^{(i)}(x_i)\,  h^2_i(x_i)\, w^{(i)}_0(x_i,t) \big) = - 2 h_i(x_i)\,  \widehat{\varphi}^{(i)}(x_i, t), \quad (x_i, t) \in I_\varepsilon^{(i)} \times (0, T),
\end{equation}
where $\partial_t{w}:=\frac{\partial w}{\partial t}, \ \partial_{x_i}{w}:=\frac{\partial w}{\partial x_i},$ and
\begin{equation}\label{hat_phi}
  \widehat{\varphi}^{(i)}(x_i, t) := \frac{1}{2\pi h_i(x_i)} \int_{\partial \Upsilon_i(x_i)} \varphi^{(i)}(\bar{\xi}_i,  x_i, t)\, d\sigma_{\bar{\xi}_i}.
\end{equation}

Let there be a solution of the equation \eqref{lim_0} (we will show its existence below). Then there is a unique solution to the inhomogeneous Neumann problem consisting of the equations \eqref{eq_1}, \eqref{bc_1} and \eqref{uniq_1}  $(k=1).$

Repeating this procedure for \eqref{eq_3} and \eqref{bc_3} and taking into account \eqref{uniq_1}, we get
\begin{multline}\label{lim_2}
  h^2_i(x_i)\, \partial_{t}{w}^{(i)}_k(x_i,t) + \partial_{x_i}\big( v_i^{(i)}(x_i)\,  h^2_i(x_i)\, w^{(i)}_k(x_i,t) \big)
\\
= a_{ii}^{(i)} \partial_{x_i} \big(h^2_i(x_i)\, \partial_{x_i} w_{k-1}^{(i)}(x_i,t)\big)
-  v_i^{(i)}(x_i) \partial_{x_i}\big(h^2_i(x_i)\big) \,  \widehat{u}^{(i)}_k(x_i, t)
\\
+ a_{ii}^{(i)} \partial_{x_i}\big(h^2_i(x_i)\big) \,   \partial_{x_i}\widehat{u}^{(i)}_{k-1}(x_i, t) -  2 h_i(x_i)\, \eta_k^{(i)}(x_i) \,  \widehat{\varphi}^{(i)}(x_i, t), \quad (x_i, t) \in I_\varepsilon^{(i)} \times (0, T)),
\end{multline}
where
\begin{equation}\label{hat_u_k}
 \widehat{u}^{(i)}_k(x_i, t) := \frac{1}{2\pi h_i(x_i)} \int_{\partial \Upsilon_i(x_i)} {u}^{(i)}_k(x_i, t)(x_i, \bar{\xi}_i, t)\, d\sigma_{\bar{\xi}_i}.
\end{equation}
Again, let there be a solution of the equation \eqref{lim_2}. Then there is a unique solution to the inhomogeneous Neumann problem consisting of the equations \eqref{eq_3}, \eqref{bc_3} and \eqref{uniq_1}.

Thus, we can uniquely determine solutions $\{{u}^{(i)}_k\}_{k\in \Bbb N}$ of the recurrent relations of the inhomogeneous Neumann problems knowing solutions $\{{w}^{(i)}_k\}_{k\in \Bbb N_0}$ of the recurrent relations \eqref{lim_0} and  \eqref{lim_2} of the partial differential equations of the first order.

\begin{remark}\label{r_2_1}
Since the functions $\{\varphi^{(i)}\}_{i=1}^3$ and $\{\overline{V}^{(i)}\}_{i=1}^3$  have compact supports with respect to the corresponding longitudinal variable and the functions $\{h_i\}_{i=1}^3$ are constant in neighborhoods
of the ends of segment $[0, \ell_i]$, the coefficients $\{u_k^{(i)}\}$ vanish in the corresponding neighborhoods
of the ends of $[0, \ell_i].$
\end{remark}

\subsection{Node-layer part of the asymptotics. The limit problem}\label{subsec_Inner_part}
To uniquely define the coefficients $\{w_k^{(i)}\}$ we should launch
the inner part of the asymptotics for the solution to problem ~(\ref{probl}). For this purpose we pass to the variables $\xi=\frac{x}{\varepsilon}.$ Letting $\varepsilon$ to $0,$ we see  that the domain $\Omega_\varepsilon$ is transformed into the unbounded domain $\Xi$ that  is the union of the domain~$\Xi^{(0)}$ and three semibounded cylinders
$$
\Xi^{(i)}
 =  \{ \xi=(\xi_1,\xi_2,\xi_3)\in\Bbb R^3 \ :
    \quad  \ell_0<\xi_i<+\infty,
    \quad |\overline{\xi}_i|<h_i(0) \},
\qquad i\in \{1,2,3\},
$$
i.e., $\Xi$ is the interior of the set $\bigcup_{i=0}^3\overline{\Xi^{(i)}}$ (see Fig.~\ref{Fig-5}).

\begin{figure}[htbp]
\begin{center}
\includegraphics[width=7cm]{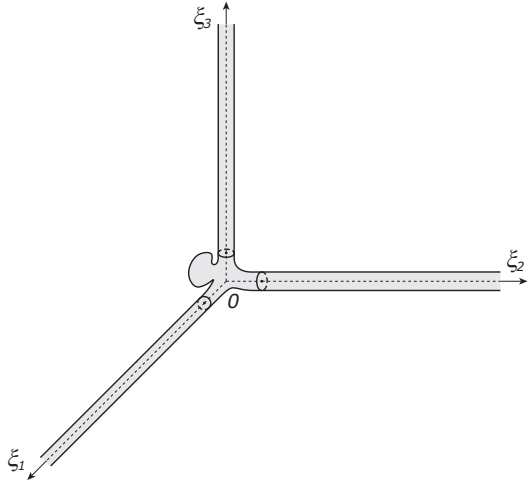}
\vskip - 10pt
\caption{Domain $\Xi$}\label{Fig-5}
\end{center}
\end{figure}

For parts of the   boundary of the domain $\Xi$ we  introduce notation
$$
\Gamma_i = \{ \xi\in\Bbb R^3 \colon \ \  \ell_0 <\xi_i<+\infty, \quad |\overline{\xi}_i|=h_i(0)\}, \quad i\in \{1,2,3\}, \quad \text{and} \quad
\Gamma_0 = \partial\Xi \backslash \Big(\bigcup_{i=1}^3 \Gamma_i \Big).
$$

In a neighborhood of the node $\Omega^{(0)}_\varepsilon$ the following ansatz is proposed:
 \begin{equation}\label{inner_part}
\mathfrak{N}_\varepsilon :=\sum\limits_{k=0}^{+\infty}\varepsilon^k N_k\left(\frac{x}{\varepsilon},t\right).
\end{equation}
Substituting $\mathfrak{N}_\varepsilon$  into the differential equation and  boundary conditions of problem ~(\ref{probl}) , collecting terms of the same powers of $\varepsilon$ and taking into account the assumptions {\bf A2} and  Remark~\ref{r_2_1},  we get for each $k\in \Bbb N_0$ the following relations:
\begin{equation}\label{N_0_prob}
\left\{\begin{array}{rcll}
  -   \mathrm{div}_\xi \big( \mathbb{D}^{(0)}(\xi) \nabla_\xi {N}_0 \big) +
  \overrightarrow{V}(\xi)\cdot \nabla_\xi {N}_0 & = &  0, &  \xi \in\Xi^{(0)},
\\[2mm]
-\sigma_\xi(N_0(\xi,t))  &=&  \varphi^{(0)}(\xi,t), &   \xi \in \Gamma_0,
\\[2mm]
 -   \mathrm{div}_\xi \big( \mathbb{D}^{(i)}(\xi) \nabla_\xi{N}_0 \big) +
  \mathrm{v}_i \, \partial_{\xi_i}{N}_0  & = &  0, &
     \xi \in\Xi^{(i)}, \ \  i\in \{1,2,3\},
\\[2mm]
\sigma_\xi(N_0)  &=&  0, &
   \xi \in \Gamma_i, \ \ i\in \{1,2,3\},
\\[2mm]
N_0(\xi,t)     \  \sim \   w^{(i)}_{0}(0,t)   &\text{as}  &\xi_i \to +\infty, &
   \xi  \in \Xi^{(i)}, \ \  i\in \{1,2,3\};
\end{array}\right.
\end{equation}
and for $k\in \Bbb N$
\begin{equation}\label{N_k_prob}
\left\{\begin{array}{rcll}
-   \mathrm{div}_\xi \big( \mathbb{D}^{(0)}(\xi) \nabla_\xi {N}_k(\xi,t) \big) +
  \overrightarrow{V}(\xi)\cdot \nabla_\xi {N}_k(\xi,t) & = &  - \dot{{N}}_{k-1}(\xi,t), &  \xi \in\Xi^{(0)},
\\[2mm]
 -   \mathrm{div}_\xi \big( \mathbb{D}^{(i)}(\xi) \nabla_\xi{N}_k \big) +
  \mathrm{v}_i \, \partial_{\xi_i}{N}_k  & = &  - \dot{{N}}_{k-1}(\xi,t), &
     \xi \in\Xi^{(i)}, \ \  i\in \{1,2,3\},
\\[2mm]
\sigma_\xi(N_k) &=&  0, &
   \xi \in \partial\Xi,
\\[2mm]
N_k(\xi,t)     \  \sim \   w^{(i)}_{k}(0,t) + \Psi^{(i)}_{k}(\xi_i,t)      &\text{as}  &\xi_i \to +\infty, &
    \xi  \in \Xi^{(i)}, \ \  i\in \{1,2,3\},
 \end{array}\right.
\end{equation}
where the variable $t$ is considered as a parameter in these problems,  $\sigma_\xi(N_k) := \boldsymbol{\nu}_{\xi} \cdot \big({\mathbb{D}} \nabla_{\xi} N_k\big) ,$
$\boldsymbol{\nu}_{\xi}$ is the outward unit normal to $\partial \Xi,$
 $\mathbb{D}(\xi)= \mathbb{D}^{(0)}(\xi) := \mathbb{D}_\varepsilon^{(0)}(x)\big|_{x= \varepsilon \xi}$ if $\xi \in \Xi^{(0)},$
 $$
\mathbb{D}(\xi) = \mathbb{D}^{(1)}(\overline{\xi}_1) :=
\left(
\begin{matrix}
  a^{(1)}_{11} & 0 & 0 \\[2mm]
  0 & a^{(1)}_{22}(\overline{\xi}_1)  &  a^{(1)}_{23}(\overline{\xi}_1) \\[2mm]
  0 &  a^{(1)}_{32}(\overline{\xi}_1) &  a^{(1)}_{33}(\overline{\xi}_1)
\end{matrix}
\right), \quad \xi \in \Xi^{(1)},
$$

$$
\mathbb{D}(\xi)  = \mathbb{D}^{(2)}(\overline{\xi}_2) :=
\left(
\begin{matrix}
  a^{(2)}_{11}(\overline{\xi}_2) & 0 & a^{(2)}_{13}(\overline{\xi}_2) \\[2mm]
  0 & a^{(2)}_{22} & 0 \\[2mm]
  a^{(2)}_{31}(\overline{\xi}_2) & 0 & a^{(2)}_{33}(\overline{\xi}_2)
\end{matrix}
\right),  \quad \xi \in \Xi^{(2)},
$$

$$
\mathbb{D}(\xi) = \mathbb{D}^{(3)}(\overline{\xi}_3) :=
\left(
\begin{matrix}
  a^{(3)}_{11}(\overline{\xi}_3) & a^{(3)}_{12}(\overline{\xi}_3) & 0 \\[2mm]
  a^{(3)}_{21}(\overline{\xi}_3) & a^{(3)}_{22}(\overline{\xi}_3) & 0 \\[2mm]
  0 & 0 & a^{(3)}_{33}
\end{matrix}
\right), \quad \xi \in \Xi^{(3)},
$$
$
\overrightarrow{V}(\xi) = \overrightarrow{V}_\varepsilon^{(0)}(x)\big|_{x= \varepsilon \xi}
$ if $\xi \in \Xi^{(0)},$
$$
\overrightarrow{V}(\xi) = \left( \mathrm{v}_1, \, 0,\, 0\right) \ \  \text{if}  \ \ \xi \in \Xi^{(1)},
\quad
\overrightarrow{V}(\xi) = \left(0, \,  \mathrm{v}_2,\, 0\right) \ \ \text{if} \ \ \xi \in \Xi^{(2)},
\quad
\overrightarrow{V}(\xi) = \left( 0,\, 0, \, \mathrm{v}_3\right) \ \ \text{ if} \ \ \xi \in \Xi^{(3)},
$$
and  in view of the assumptions {\bf A2} we have that $\overrightarrow{V} \cdot \boldsymbol{\nu}_\xi =0$ on $\Gamma_i,$ $ i\in \{0,1,2,3\}.$

Relations in the last lines both of \eqref{N_0_prob} and \eqref{N_k_prob}  appear by matching the regular and inner asymptotics in a neighborhood of the node, namely the asymptotics of the terms $\{N_k\}$ as $\xi_i \to +\infty$ have to coincide with the corresponding asymptotics of  terms of the regular expansions (\ref{regul}) as $x_i =\varepsilon \xi_i \to +0, \ i=1,2,3,$ respectively.
Expanding each term of the regular asymptotics in the Taylor series at the points $x_i=0, \ i=1,2,3,$
and collecting the coefficients of the same powers of $\varepsilon,$  we get
\begin{equation}\label{Psi_k}
\Psi_{0}^{(i)} \equiv 0, \qquad
\Psi_{k}^{(i)}(\xi_i,t)
 =   \sum\limits_{j=1}^{k} \dfrac{\xi_i^j}{j!}
     \dfrac{\partial^j w_{k-j}^{(i)}}{\partial x_i^j} (0,t),
\quad  i=1,2,3, \ \ k \in \Bbb N.
\end{equation}

We look for a solution to the problem \eqref{N_0_prob} in the form
\begin{equation}\label{new-solution_0}
N_0(\xi,t) = \sum\limits_{i=1}^3w^{(i)}_{k}(0,t) \,\chi_{\ell_0}(\xi_i) + \widetilde{N}_0(\xi,t),
\end{equation}
where $ \chi_{\ell_0} \in C^{\infty}(\Bbb{R})$ is a cut-off function such that
$\ 0\leq \chi_{\ell_0} \leq1,$  $\chi_{\ell_0}(t) =0$ if $t \leq  1+\ell_0$  and
$\chi_{\ell_0}(t) =1$ if $t \geq  2+\ell_0.$
Then $\widetilde{N}_0$ has  to be  a  solution to the problem
\begin{equation}\label{tilda_N_0_prob}
\left\{\begin{array}{rcll}
  -   \mathrm{div}_\xi \big( \mathbb{D}^{(0)}(\xi) \nabla_\xi \widetilde{N}_0 \big) +
  \overrightarrow{V}(\xi)\cdot \nabla_\xi \widetilde{N}_0
  & = & 0, & \quad
    \xi \in\Xi^{(0)},
\\[2mm]
-\sigma_\xi(\widetilde{N}_0(\xi,t))  &=&  \varphi^{(0)}(\xi,t), &
   \quad \xi \in \Gamma_0,
\\[2mm]
 -   \mathrm{div}_\xi \big( \mathbb{D}^{(i)}(\xi) \nabla_\xi \widetilde{N}_0(\xi,t) \big) +
  \mathrm{v}_i \, \partial_{\xi_i} \widetilde{N}_0(\xi,t)  & = & f_0^{(i)}(\xi_i,t), &
    \quad \xi \in\Xi^{(i)},
\\[2mm]
\sigma_\xi(\widetilde{N}_0)  &=&  0, & \quad
   \xi \in \Gamma_i, \quad i\in \{1,2,3\},
\end{array}\right.
\end{equation}
and has to satisfy the conditions:
\begin{equation}\label{junc_probl_general+cond_0}
   \widetilde{N}_0(\xi,t)  \rightarrow  0
   \quad \text{as} \quad \xi_i \to +\infty, \ \  \xi  \in \Xi^{(i)}, \quad i\in \{1,2,3\},
\end{equation}
where $ \partial_{\xi_i} = \frac{\partial}{\partial_{\xi_i}},$
\begin{equation}\label{F_1-0}
f_0^{(i)}(\xi_i,t)  =  a^{(i)}_{ii} \, w^{(i)}_0(0,t) \, \chi''_{\ell_0}(\xi_i) - \mathrm{v}_i  \, w^{(i)}_0(0,t) \, \chi'_{\ell_0}(\xi_i) .
\end{equation}

A solution to the problem \eqref{N_k_prob} is sought in the form
\begin{equation}\label{new-solution_k}
N_k(\xi,t) = \sum\limits_{i=1}^3\big(w^{(i)}_{k}(0,t) + \Psi^{(i)}_{k}(\xi_i,t) \big) \,\chi_{\ell_0}(\xi_i) + \widetilde{N}_k(\xi,t).
\end{equation}
Then $\widetilde{N}_k$ has  to be  a  solution to the problem
\begin{equation}\label{tilda_N_k_prob}
\left\{\begin{array}{rcll}
  -   \mathrm{div}_\xi \big( \mathbb{D}^{(0)}(\xi) \nabla_\xi \widetilde{N}_k(\xi,t) \big) +
  \overrightarrow{V}(\xi)\cdot \nabla_\xi \widetilde{N}_k(\xi,t)
  & = &  - \dot{{N}}_{k-1}(\xi,t), &
    \xi \in\Xi^{(0)},
\\[2mm]
 -   \mathrm{div}_\xi \big( \mathbb{D}^{(i)}(\xi) \nabla_\xi \widetilde{N}_k \big) +
  \mathrm{v}_i \, \partial_{\xi_i} \widetilde{N}_k  & = & f_k^{(i)}   - \dot{{N}}_{k-1}, &
     \xi \in\Xi^{(i)},
\\[2mm]
\sigma_\xi(\widetilde{N}_k)  &=&  0, &
   \xi \in \Gamma_i, \ \  i\in \{0,1,2,3\},
\end{array}\right.
\end{equation}
and has to satisfy the conditions:
\begin{equation}\label{junc_probl_general+cond}
   \widetilde{N}_k(\xi)  \rightarrow  0
   \quad \text{as} \quad \xi_i \to +\infty, \ \  \xi  \in \Xi^{(i)}, \quad i\in \{1,2,3\}.
\end{equation}
Here
\begin{equation}\label{F_1}
\begin{array}{rcl}
f_1^{(i)}(\xi_i,t) &= & \ a^{(i)}_{ii} \Big( w^{(i)}_1(0,t) \, \chi''_{\ell_0}(\xi_i) + \dfrac{\partial w_0^{(i)}}{\partial x_i}(0,t)\,   \Big( \big(\xi_i \chi_{\ell_0}^{\prime}(\xi_i)\big)^\prime
 +  \chi_{\ell_0}^{\prime}(\xi_i) \Big) \Big) \notag
\\[1.2ex]
&& {} -  \ \mathrm{v}_i  \Big( w^{(i)}_1(0,t) \, \chi'_{\ell_0}(\xi_i) + \dfrac{\partial w_0^{(i)}}{\partial x_i}(0,t) \,  \big(\xi_i \chi_{\ell_0}(\xi_i)\big)^\prime\Big), \end{array}
\end{equation}
and
\begin{equation}\label{F_k}
\begin{array}{rcl}
f_k^{(i)}(\xi_i,t)  &= & \ a^{(i)}_{ii} \Big( w^{(i)}_k(0,t) \, \chi''_{\ell_0}(\xi_i) + \Big(\Psi_{k}^{(i)}(\xi_i,t)\chi^\prime_i(\xi_i)\Big)^\prime + \Big(\Psi_{k}^{(i)}(\xi_i,t)\Big)^\prime \chi^\prime_i(\xi_i)
+ \Big(\Psi_{k}^{(i)}(\xi_i,t)\Big)^{\prime\prime}  \chi_i(\xi_i)\Big) \notag
\\[1.2ex]
 && {}-  \ \mathrm{v}_i  \Big( w^{(i)}_k(0,t) \, \chi'_{\ell_0}(\xi_i)  + \big(\Psi_{k}^{(i)}(\xi_i,t) \, \chi_{\ell_0}(\xi_i)\big)^\prime\Big).
 \end{array}
\end{equation}

Boundary value problems in unbounded domains are well studied (see e.g. \cite{Ber-Nir_1990, Kon-Ole_1983,Koz-Maz-Ros_97,Na-Pla,Naz99,Ole_book_1996}).
We will use the general approach proposed in \cite[\S 2.2]{Ole_book_1996} and \cite[\S 3]{Naz99}, which was realised for problems like
\eqref{tilda_N_0_prob} and \eqref{tilda_N_k_prob} in \cite{Mel-Klev_AA-2022}.

Let us define for $\beta >0 $ the weighted Sobolev space $\mathcal{H}_\beta$ as the set of all functions from $H^1(\Xi)$ with  finite norm
$$
\|u\|_{\beta} := \bigg(\int_{\Xi} \varrho(\xi)\Big(|\nabla_\xi u|^2 + |u|^2 \Big) d\xi  \bigg)^{1/2}.
$$
Here $\varrho$ is a smooth function such that
$$
\varrho(\xi) = \left\{
                 \begin{array}{ll}
                   1, & \ \ \xi \in \Xi^{(0)},
\\
                   e^{\beta \xi_i} , &\ \  \xi_i \ge 2 \ell_0, \ \ \xi \in \Xi^{(i)}, \ i\in\{1, 2, 3\}.
                 \end{array}
               \right.
$$

Consider a problem
\begin{equation}\label{tilda_N}
\left\{\begin{array}{rcll}
  -   \mathrm{div}_\xi \big( \mathbb{D}^{(0)}(\xi) \nabla_\xi \widetilde{N} \big) +
 \overrightarrow{V}(\xi) \cdot \nabla_\xi \widetilde{N}
  & = &  F^{(0)}, &
    \xi \in\Xi^{(0)},
\\[2mm]
-\sigma_\xi(\widetilde{N})  &=&  \Psi^{(0)}, &
    \xi \in \Gamma_0,
\\[2mm]
 -   \mathrm{div}_\xi \big( \mathbb{D}^{(i)}(\xi) \nabla_\xi \widetilde{N}\big) +
  \mathrm{v}_i \, \partial_{\xi_i} \widetilde{N}  & = & F^{(i)}, &
     \xi \in\Xi^{(i)}, \ \  i\in \{1,2,3\},
\\[2mm]
\sigma_\xi(\widetilde{N})  &=&  0, &
   \xi \in \Gamma_i, \ \  i\in \{1,2,3\}.
\end{array}\right.
\end{equation}

Based on results of \cite[Lemma 3.1]{Mel-Klev_AA-2022} we have the following  statement.
\begin{proposition}\label{Prop-2-1} Let the functions on the right-hand sides of the problem \eqref{tilda_N} be given satisfying the following conditions:
$F^{(0)} \in L^2(\Xi^{(0)}),$ $\Psi^{(0)} \in L^2(\Gamma_0),$ and for all $i\in \{1, 2, 3\}$
$$
  \int_{\Xi^{(i)}} e^{\beta \xi_i} \, \big(F^{(i)}(\xi)\big)^2 \,  d\xi < +\infty \quad (\beta > 0).
$$

Then the problem \eqref{tilda_N} has a unique solution in the space $\mathcal{H}_\beta$ if and only if
the equality
\begin{equation}\label{cong_cond_g}
 \sum_{i=0}^{3}\int_{\Xi^{(i)}} F^{(i)}(\xi) \,  d\xi = \int_{\Gamma_0} \Psi^{(0)}(\xi) \,  d\sigma_\xi
\end{equation}
holds.
\end{proposition}

The equality \eqref{cong_cond_g} for the problem \eqref{tilda_N_0_prob} is equivalent to
\begin{equation}\label{cong_cond}
  \sum_{i=1}^{3}  h_i^2(0) \,\mathrm{v}_i \,  w_0^{(i)}(0, t) = \breve{\varphi}^{(0)}(t),
\end{equation}
where
\begin{equation}\label{breve_phi}
\breve{\varphi}^{(0)}(t) := - \frac{1}{\pi}\int_{\Gamma_0} \varphi^{(0)}(\xi,t) \, d\sigma_\xi.
  \end{equation}

\subsubsection{Limit problem $($the first case: $v^{(1)}_1 < 0, v^{(2)}_2 < 0, v^{(3)}_3 > 0)$. Existence of solutions $\{u^{(i)}_1\}, \ N_0$}\label{sub_limit_problem}

When constructing an asymptotic approximation, it is not always immediately clear which boundary condition or initial condition the regular asymptotic coefficients must satisfy. In our case we have three partial differential equations \eqref{lim_0} of the first order, each of them is on the corresponding interval $I_i:=\{x: \ x_i\in (0,\ell_i), \ \overline{x_i}=(0,0)\}$ of the graph  $\mathcal{I};$ and they are connected by the Kirchhoff transmission condition \eqref{cong_cond}. In addition, we would like these solutions to satisfy three boundary conditions at $x_1 = \ell_1,$ $x_2 = \ell_2,$ $x_3 = \ell_3,$ respectively, and the  initial condition at $t= 0.$ Obviously, there is no solution to satisfy all these conditions.

Thus, we propose the  following
\textit{limit problem to \eqref{probl} for the first case}:
\begin{equation}\label{limit_prob}
 \left\{\begin{array}{rcll}
h^2_i(x_i)\, \partial_t{w}^{(i)}_0 + \partial_{x_i}\big( v_i^{(i)}(x_i)\,  h^2_i(x_i)\, w^{(i)}_0 \big) &=& - 2 h_i(x_i)\,  \widehat{\varphi}^{(i)}(x_i, t),& (x_i, t) \in I_i \times (0, T),  \ \ i \in \{1,2,3\},
 \\[3mm]
\sum_{i=1}^{3}  \mathrm{v}_i \, h_i^2(0) \,  w_0^{(i)}(0, t) & = & \breve{\varphi}^{(0)}(t),& t \in (0, T),
    \\[2mm]
    w_0^{(i)}(\ell_i,  t) & = & q_i(t), & t \in [0, T],   \ \ i \in \{1,2\},
\\[2mm]
    w_0^{(i)}(x_i,  0) & = & 0, & x_i  \in [0, \ell_i],  \ \ i \in \{1,2,3\}.
  \end{array}\right.
\end{equation}
This problem does not take into account  the boundary condition at $x_3=\ell_3.$

Let us prove its solvability. First, we consider the model mixed problem
\begin{equation}\label{prob_1}
 \left\{\begin{array}{rcll}
 \partial_t U(y,t) + \Lambda(y)\, \partial_{y} U(y,t)  &=& \Psi(y,t),& (y, t) \in (0, \ell) \times (0, T) ,
    \\[2mm]
    U(0,  t) & = & q(t), & t \in [0, T],
    \\[2mm]
    U(y, 0) & = & 0, & t \in [0, \ell],
    \end{array}\right.
\end{equation}
where the  coefficient $\Lambda(y), \ y\in [0, \ell],$  the given functions $\Psi(y,t), \ (y, t) \in [0, \ell] \times [0, T],$ and $q(t), \ t\in[0,T],$  are infinitely differentiable functions on their domains of definition,   and in addition,  $\Lambda >0.$

According to the theory of  first-order partial differential equations (see, e.g.,\cite{Myshkis_1960}), the mixed problem \eqref{prob_1} has a unique classical solution under special matching conditions. In our case, we can find a representation of the solution using the first integral method and derive matching conditions that will ensure the greater smoothness of the solution needed in the seque.l
The characteristic
$$
t = \mathcal{V}(y):= \int_{0}^{y} \frac{d\xi}{\Lambda(\xi)}, \quad y\in[0, \ell],
$$
of the corresponding characteristic system
$$
\frac{dy}{dt} = \Lambda (y), \quad \frac{du}{dt} = \Psi(y,t),
$$
divides the rectangle $(0, \ell) \times (0, T)$ into two domains, namely
$$
\mathfrak{D}_1 := \{(y,t)\colon y\in (0,\ell), \ \ t \in (0, \mathcal{V}(y))\} \quad \text{and} \quad \mathfrak{D}_2 := \big\{(0, \ell) \times (0, T)\big\} \setminus \overline{\mathfrak{D}_1}.
$$

Therefore,
\begin{equation}\label{model_solution}
  U(y,t) =
  \left\{\begin{array}{ll}
\displaystyle{ \int\limits_{0}^{t} \Psi\big(\mathcal{V}^{-1}(\tau + \mathcal{V}(y) - t) , \tau \big)\, d\tau,} & (y,t) \in \mathfrak{D}_1,
    \\[2mm]
    \displaystyle{ q\big(t - \mathcal{V}(y)\big) + \int\limits_{0}^{y} \frac{\Psi\big(\eta, \mathcal{V}(\eta) +t - \mathcal{V}(y)\big) }{\Lambda(\eta)} \, d\eta,}
 & (y,t) \in \mathfrak{D}_2,
    \end{array}\right.
\end{equation}
is a classical solution to the problem \eqref{prob_1} if it and its partial derivations of the first-order are continuous  on the curve $t= \mathcal{V}(y), \ y\in [0, \ell];$ as a result, we get the  matching conditions
\begin{equation}\label{match_cond_1}
q(0) =0 \quad \text{and} \quad q'(0) = \Psi(0,0).
\end{equation}
In \eqref{model_solution},  $\mathcal{V}^{-1}$ is the inverse function; it exists because $\mathcal{V}$ is strictly increasing on $[0, \ell].$
Also, we note that $\mathcal{V}(0)=0$ and $\mathcal{V}(y)>0$ for $y\in (0, \ell].$

For each $i \in \{1,2\}$ the problem
\begin{equation}\label{limit_prob_1_2}
 \left\{\begin{array}{rcll}
h^2_i(x_i)\, \partial_t{w}^{(i)}_0(x_i, t) + \partial_{x_i}\big( v_i^{(i)}(x_i)\,  h^2_i(x_i)\, w^{(i)}_0(x_i, t) \big) &=& - 2 h_i(x_i)\,  \widehat{\varphi}^{(i)}(x_i, t),& (x_i, t) \in I_i \times (0, T),
   \\[2mm]
    w_0^{(i)}(\ell_i,  t) & = & q_i(t), & t \in [0, T],
\\[2mm]
    w_0^{(i)}(x_i,  0) & = & 0, & x_i  \in [0, \ell_i],
  \end{array}\right.
\end{equation}
is reduced to \eqref{prob_1}  using the substitutions: $U(x_i,t) = v_i^{(i)}(x_i)\,  h^2_i(x_i)\, w^{(i)}_0(x_i, t)$ and $y= \ell_i - x_i.$ In this case
$$
\Psi(y,t) = - 2 v_i^{(i)}(\ell_i - y)\, h_i(\ell_i - y)\,  \widehat{\varphi}^{(i)}(\ell_i - y, t), \quad \Lambda(y) = - v_i^{(i)}(\ell_i - y) > 0, \quad q(t) = v_i^{(i)}(\ell_i)\,  h^2_i(\ell_i)\, q_i(t),
$$
and the unique classical solution is given by the formula \eqref{model_solution} if
the  matching conditions
\begin{equation*}
q_i(0) =0 \quad \text{and} \quad q'_i(0) =  - \frac{2}{h_i(\ell_i)} \, \widehat{\varphi}^{(i)}(\ell_i,0)
\end{equation*}
are satisfied. They are performed due to the assumptions \eqref{match_conditions}.

In addition, since the function ${\varphi}^{(i)}$  vanishes  uniformly with respect to $t$  and $\overline{\xi}_i$ in neighborhoods of the ends  of the closed interval  $[0, \ell_i]$ (the same for the function $\widehat{\varphi}^{(i)}$; it is determined in \eqref{hat_phi}), the problem \eqref{limit_prob_1_2} has the  infinitely differentiable solution  if
\begin{equation}\label{match_cond_3}
   \frac{d^k q_i }{dt^k}(0) =  0, \quad k\in \Bbb N.
\end{equation}

The solution to the problem \eqref{limit_prob_1_2}  is given by the formula
$$
 w^{(i)}_0(x_i, t) = \frac{1}{v_i^{(i)}(x_i)\,  h^2_i(x_i)}
$$
\begin{equation}\label{model_solution_w_0}
  \times
  \left\{\begin{array}{ll}
\displaystyle{ -2 \int\limits_{0}^{t} v_i^{(i)}(\zeta)\,  h^2_i(\zeta) \,  \widehat{\varphi}^{(i)}(\zeta, \tau)\Big|_{\zeta = \ell_i - \mathcal{V}^{-1}(\tau + \mathcal{V}(\ell_i - x_i) - t)} \, d\tau,} & (x_i,t) \in \mathfrak{B}_1,
    \\[2mm]
    \displaystyle{v_i^{(i)}(\ell_i)\,  h^2_i(\ell_i)\, q_i\big(t - \mathcal{V}(\ell_i - x_i)\big) + 2\int\limits_{0}^{\ell_i - x_i}
    h_i(\ell_i - \eta) \, \widehat{\varphi}^{(i)}\big(\ell_i - \eta,   \mathcal{V}(\eta) + t -  \mathcal{V}(\ell_i - x_i)\big) \, d\eta,}
 & (x_i,t) \in \mathfrak{B}_2,
    \end{array}\right.
\end{equation}
where
$$
\mathfrak{B}_1 := \big\{(x_i,t)\colon x_i\in (0,\ell_i), \ \ t \in \big(0, \mathcal{V}(\ell_i - x_i)\big)\big\} \quad \text{and} \quad \mathfrak{B}_2 := \big\{(0, \ell_i) \times (0, T) \big\}\setminus \overline{\mathfrak{B}_1}.
$$

 The problem for ${w}^{(3)}_0$ is as follows
\begin{equation}\label{limit_prob_w_3}
 \left\{\begin{array}{c}
h^2_3(x_3)\, \partial_t{w}^{(3)}_0(x_3, t) + \partial_{x_3}\big( v_3^{(3)}(x_3)\,  h^2_3(x_3)\, w^{(3)}_0(x_3, t) \big) = - 2 h_3(x_3)\,  \widehat{\varphi}^{(3)}(x_3, t), \qquad (x_3, t) \in I_3 \times (0, T),
 \\[2mm]
 {w}^{(3)}_0(0, t)   =   \displaystyle{\frac{1}{\mathrm{v}_3 \, h_3^2(0)}\Big(\breve{\varphi}^{(0)}(t)  -
\sum_{i=1}^{2}  \mathrm{v}_i \, h_i^2(0) \,  w_0^{(i)}(0, t) \Big)}, \qquad t \in (0, T),
 \\[4mm]
    w_0^{(3)}(x_3,  0)  =  0, \qquad x_3  \in [0, \ell_3].
  \end{array}\right.
\end{equation}
We see that the boundary condition for ${w}^{(3)}_0$ at $x_3 =0$ takes into account the Kirchhoff transmission condition \eqref{cong_cond} in the limit problem \eqref{limit_prob}.
After introducing a new function $U(y,t):=   v_3^{(3)}(y)\,  h^2_3(y)\, w^{(3)}_0(y, t) v(y),$ $ (y=x_3),$  the problem \eqref{limit_prob_w_3} is
reduced to the model  problem \eqref{prob_1} with
$$
\Psi(y,t) = - 2v_3^{(3)}(y)\,  h_3(y)\,  \widehat{\varphi}^{(3)}(y, t), \quad \Lambda(y) = v_3^{(3)}(y) > 0, \quad q(t) =
\breve{\varphi}^{(0)}(t) - \sum_{i=1}^{2}  \mathrm{v}_i \, h_i^2(0) \,  w_0^{(i)}(0, t).
$$
By virtue of \eqref{match_conditions} for $\varphi^{(0)}$ and the initial conditions for $w_0^{(1)}$ and $w_0^{(2)},$  the first matching condition in \eqref{match_cond_1} is satisfied. The second matching condition reads as follows
\begin{equation}\label{match_cond_4}
\frac{d\breve{\varphi}^{(0)}}{dt}(0) =  \sum_{i=1}^{2}  \mathrm{v}_i \, h_i^2(0) \, \partial_t w_0^{(i)}(0, t)\big|_{t=0},
\end{equation}
where the function $\breve{\varphi}^{(0)}$  is defined in \eqref{breve_phi}.
 It follows from \eqref{model_solution_w_0} and the assumptions {\bf A3} that
$$
\partial_t w_0^{(i)}(0, 0) = - 2  \, \widehat{\varphi}^{(i)}(0, 0) = 0.
$$
Therefore, the problem \eqref{limit_prob_w_3} has a unique classical solution if $\partial_t \breve{\varphi}^{(0)}(0) = 0.$ Moreover, it also follows from \eqref{model_solution_w_0} that $\partial^k_t w_0^{(i)}(0, 0) = 0$ for each $k\in \Bbb N$  if
\begin{equation}\label{match_cond_5+}
\frac{\partial^k\varphi^{(i)}(\xi,t)}{\partial t^k}\Big|_{t=0}= 0, \quad  k\in \Bbb N.
\end{equation}
Thus, for the infinite differentiability of $w_0^{(3)}$, we will require the conditions \eqref{match_cond_5+} and
\begin{equation}\label{match_cond_5}
\frac{\partial^k\varphi^{(0)}(\xi,t)}{\partial t^k}\Big|_{t=0}= 0, \quad \xi  \in \Gamma_0, \quad k\in \Bbb N.
\end{equation}
\begin{remark}\label{differ}
Since for each $k\in \Bbb N$ the coefficient $w_k^{(i)}$ is determined through the second derivative $w_{ k-1}^{(i)}$ (see \eqref{lim_2} and below), this leads  to the need for infinite differentiability of the solution $w_0^{(i)}.$
Therefore, we additionally assume that  the functions $\{h_i\}_{i=1}^3,$ all coefficients of the problem \eqref{probl}, the functions $\{\varphi^{(i)}, \ q_i\}_{i=1}^3$  belong to the class $C^\infty$ in their domains of definition, the function $\varphi^{(0)}$ is infinitely differentiable in $t\in [0, T],$ and relations \eqref{match_cond_3}, \eqref{match_cond_5+} and \eqref{match_cond_5} hold.
  \end{remark}

As a result, the limit problem \eqref{limit_prob} has a unique  solution and the solvability condition \eqref{cong_cond} for the problem \eqref{tilda_N_0_prob} is satisfied. This means that there exists a unique solution $N_0$ to the problem \eqref{N_0_prob}. The solvability condition \eqref{lim_0} also holds for the problem \eqref{eq_1} and \eqref{bc_1} and therefore, there  exists a unique solution $u_1^{(i)}$ that satisfies the condition  \eqref{uniq_1}. In addition,
\begin{equation}\label{zero_t_0}
  N_0\big|_{t=0} \equiv \widetilde{N}_0\big|_{t=0} \equiv 0, \quad u_1^{(i)}\big|_{t=0} \equiv 0, \quad i\in \{1, 2, 3\}.
\end{equation}

Thus, the first coefficients both of the regular asymptotics and the inner one are successfully determined.

Furthermore, we can obtain additional information above $\dot{{N}}_0 := \partial_t N_0.$ Clearly that  $\dot{\widetilde{N}}_0$ is a solution of the problem \eqref{tilda_N_0_prob} with the right-hand sides $\partial_t \varphi^{(0)}$ and $\partial_t f_0^{(i)},$ respectively, and the solvability condition
$$
 \sum_{i=1}^{3}  h_i^2(0) \,\mathrm{v}_i \,  \partial_t w_0^{(i)}(0, t) = \partial_t \breve{\varphi}^{(0)}(t)
$$
is satisfied. Since $\partial_t \varphi^{(0)},$ $\partial_t f_0^{(i)},$ and $\partial_t \breve{\varphi}^{(0)}$ vanish at $t=0,$
the function  $\dot{\widetilde{N}}_0$ also vanishes at $t=0,$ and hence $\dot{{N}}_0\big|_{t=0} \equiv 0.$ Similarly, we conclude that
$\partial^2_t N_0\big|_{t=0} \equiv 0.$

In addition, since the right-hand sides  in the problem \eqref{tilda_N_0_prob} are uniformly bounded with respect to $(\xi, t)\in \Xi\times [0, T]$ and have compact supports,  the solution to the problem~\eqref{N_0_prob} has  the following  asymptotics uniform with respect to $t\in [0, T]$:
\begin{equation}\label{rem_exp-decrease+0}
N_0(\xi,t) = w^{(i)}_{0}(0,t)  +  \mathcal{ O}(\exp(-\beta_0\xi_i))
\quad \mbox{as} \ \ \xi_i\to+\infty,  \ \  \xi  \in \Xi^{(i)},  \quad i=\{1,2,3\} \quad (\beta_0 >0).
\end{equation}
Here, the symbol $\mathcal{ O}(\cdot)$ is big O notation.

\subsubsection{Existence of solutions $\{w^{(i)}_1\}_{i=1}^3, \, \{u^{(i)}_2\}_{i=1}^3$ and $N_1$}\label{par_222}
The problem for $\widetilde{N}_1$ (see \eqref{tilda_N_k_prob} and \eqref{junc_probl_general+cond}) is as follows
\begin{equation}\label{tilda_N_1_prob}
\left\{\begin{array}{rcll}
  -   \mathrm{div}_\xi \big( \mathbb{D}^{(0)}(\xi) \nabla_\xi \widetilde{N}_1(\xi,t)  +
   \overrightarrow{V}(\xi) \cdot \nabla_\xi \widetilde{N}_1(\xi,t)
  & = &  - \dot{{N}}_{0}(\xi,t), &
    \xi \in\Xi^{(0)},
\\[2mm]
 -   \mathrm{div}_\xi \big( \mathbb{D}^{(i)}(\xi) \nabla_\xi \widetilde{N}_1 \big) +
  \mathrm{v}_i \, \partial_{\xi_i} \widetilde{N}_1  & = & f_1^{(i)}(\xi_i,t)   - \dot{{N}}_{0}(\xi,t), &
     \xi \in\Xi^{(i)},
\\[2mm]
\sigma_\xi(\widetilde{N}_1)  &=&  0, &
   \xi \in \partial\Xi,
\\[2mm]
\widetilde{N}_1(\xi)  &\rightarrow &  0
   \ \  \text{as} \ \  \xi_i \to +\infty, & \xi  \in \Xi^{(i)}, \  i\in \{1,2,3\},
\end{array}\right.
\end{equation}
where $f_1^{(i)}$ is defined by \eqref{F_1} and then rewritten in a more convenient form
\begin{align}\label{F_1+}
f_1^{(i)}(\xi_i,t)  =
 & \,  a^{(i)}_{ii} \Big( w^{(i)}_1(0,t) \, \chi''_{\ell_0}(\xi_i) + \dfrac{\partial w_0^{(i)}}{\partial x_i}(0,t)\,  \big(\xi_i \chi_{\ell_0}^{\prime}(\xi_i)\big)^\prime
 \Big) \notag
\\
&
 - \mathrm{v}_i  \,  w^{(i)}_1(0,t) \, \chi'_{\ell_0}(\xi_i) - \mathrm{v}_i   \dfrac{\partial w_0^{(i)}}{\partial x_i}(0,t) \, \xi_i\,  \chi^\prime_{\ell_0}(\xi_i)
+ a^{(i)}_{ii} \, \dfrac{\partial w_0^{(i)}}{\partial x_i}(0,t) \, \chi_{\ell_0}^{\prime}(\xi_i)
 \notag
\\
 & - \mathrm{v}_i  \, \dfrac{\partial w_0^{(i)}}{\partial x_i}(0,t)\,  \chi_{\ell_0}(\xi_i).
\end{align}
The integrals over $\Xi^{(i)}$ of  summands in the first line of \eqref{F_1+} are equal to zero;
the integrals over $\Xi^{(i)}$ of  summands in the second  line  are equal to
$$
-  \pi\, h^2_i(0) \,\mathrm{v}_i  \,  w^{(i)}_1(0,t)   - \pi\, h^2_i(0) \, \mathrm{v}_i   \dfrac{\partial w_0^{(i)}}{\partial x_i}(0,t) \int_{\ell_0+1}^{\ell_0+2} \xi_i\,  \chi^\prime_{\ell_0}(\xi_i) \, d\xi_i
+ \pi\, h^2_i(0)\, a^{(i)}_{ii} \, \dfrac{\partial w_0^{(i)}}{\partial x_i}(0,t).
 $$
The last summand in \eqref{F_1+} should be integrated with $ \dot{{N}}_{0}$ and
$$
\Big|\int_{\Xi^{(i)}} \Big(\mathrm{v}_i  \, \dfrac{\partial w_0^{(i)}}{\partial x_i}(0,t)\,  \chi_{\ell_0}(\xi_i) + \dot{{N}}_{0}(\xi,t) \Big)\, d\xi  \Big| = \Big|\int_{\Xi^{(i)}} \dot{\widetilde{N}}_{0}(\xi,t) \, d\xi  \Big|  < +\infty,
$$
since
 $\mathrm{v}_i  \, \partial_{x_i} w_0^{(i)}(0,t) + \partial_t w_0^{(i)}(0,t) =0 $ and $\dot{\widetilde{N}}_{0} \in \mathcal{H}_\beta$.

Due to Proposition~\ref{Prop-2-1}  the necessary and sufficient condition for the solvability of the problem \eqref{tilda_N_1_prob}
is as follows
$$
\sum_{i=1}^{3}  h_i^2(0) \,\mathrm{v}_i \,  w_1^{(i)}(0, t) = {\bf d}_1(t),
$$
where
\begin{align}\label{d_1}
 {\bf d}_1(t) :=&  - \frac{1}{\pi}\int_{\Xi^{(0)}} \dot{{N}}_{0}(\xi,t)\, d\xi - \frac{1}{\pi} \sum_{i=1}^{3}\int_{\Xi^{(i)}} \dot{\widetilde{N}}_{0}(\xi,t) \, d\xi
\notag
\\
 & - \sum_{i=1}^{3}h^2_i(0) \Big(\mathrm{v}_i \,  \partial_{x_i}w_0^{(i)}(0,t) \int_{\ell_0+1}^{\ell_0+2} \xi_i\,  \chi^\prime_{\ell_0}(\xi_i) \, d\xi_i -  a^{(i)}_{ii} \, \partial_{x_i}w_0^{(i)}(0,t) \Big).
\end{align}

Again as in \S~\ref{sub_limit_problem}, with the help of the formula \eqref{model_solution}  we can obtain a unique solution to the problem
\begin{equation}\label{prob_w_1}
 \left\{\begin{array}{rcll}
h^2_i(x_i)\, \partial_t{w}^{(i)}_1 + \partial_{x_i}\big( v_i^{(i)}(x_i)\,  h^2_i(x_i)\, w^{(i)}_1 \big) &=& {\bf g}_1^{(i)}(x_i,t),& (x_i, t) \in I_i \times (0, T),  \ \ i \in \{1,2,3\},
 \\[3mm]
\sum_{i=1}^{3}  \mathrm{v}_i \, h_i^2(0) \,  w_1^{(i)}(0, t) & = & {\bf d}_1(t), & t \in (0, T),
 \\[2mm]
    w_1^{(i)}(\ell_i,  t) & = & 0, & t \in [0, T],   \ \ i \in \{1,2\},
\\[2mm]
    w_1^{(i)}(x_i,  0) & = & 0, & x_i  \in [0, \ell_i],  \ \ i \in \{1,2,3\},
  \end{array}\right.
\end{equation}
where
$$
{\bf g}_1^{(i)}(x_i,t) := a_{ii}^{(i)} \partial_{x_i}\big(h^2_i(x_i)\, \partial_{x_i} w_0^{(i)}(x_i,t)\big)
-  v_i^{(i)}(x_i) \, \partial_{x_i}\big(h^2_i(x_i)\big) \,  \widehat{u}^{(i)}_1(x_i, t).
$$
Thanks to the additional assumptions made in Remark~\ref{differ}, all matching conditions needed for the infinite differentiability of the solutions $\{{w}^{(i)}_1\}_{i=1}^3$ are  satisfied.

Thus, the solvability condition  for the problem \eqref{tilda_N_1_prob} is satisfied and there exists a unique solution $N_1$ to the problem \eqref{N_k_prob} $(k=1).$
Also the solvability condition \eqref{lim_2} at $k=1$  holds for the problem  \eqref{eq_3}, \eqref{bc_3} and \eqref{uniq_1} $(k=1)$,
and therefore, there  exists a unique solution $u_{2}^{(i)}.$
 In addition, it is easy to verify that
\begin{equation*}
  N_1\big|_{t=0} \equiv \widetilde{N}_1\big|_{t=0} \equiv 0, \quad \partial_t N_1\big|_{t=0} \equiv 0, \quad u_2^{(i)}\big|_{t=0} \equiv 0, \quad i\in \{1, 2, 3\}.
\end{equation*}

Since the right-hand sides  in the problem \eqref{tilda_N_1_prob} are uniformly bounded with respect to $(\xi, t)\in \Xi\times [0, T]$ and $\dot{\widetilde{N}}_{0}$ has the  asymptotics \eqref{rem_exp-decrease+0},  the solution to the problem~\eqref{N_k_prob} $(k=1)$ has  the following  asymptotics uniform with respect to $t\in [0, T]$:
\begin{equation}\label{rem_exp-decrease+1}
N_1(\xi,t) = w^{(i)}_{1}(0,t)  +  \Psi^{(i)}_{1}(\xi_i,t) + \mathcal{ O}(\exp(-\beta_0\xi_i))
\quad \mbox{as} \ \ \xi_i\to+\infty,  \ \  \xi  \in \Xi^{(i)},  \quad i=\{1,2,3\} \quad (\beta_0 >0).
\end{equation}


\subsubsection{Existence of solutions $\{w^{(i)}_k\}_{i=1}^3, \, \{u^{(i)}_{k+1}\}_{i=1}^3$ and $N_k$, $k\ge 2$}\label{par_223}

Assume that all the coefficients $\{w^{(i)}_p\}_{i=1}^3,$ $ p\in \{0,1,\ldots, k-1\},$ $\{u^{(i)}_p\}_{i=1}^3, \, p\in \{1,\ldots, k\},$ and
$\{N_p\}_{p=0}^{k-1}$  are determined and  that they and  $\partial_t N_{k-1}$   vanish at $t=0.$

Due to Proposition~\ref{Prop-2-1}  the necessary and sufficient condition for the solvability of the problem \eqref{tilda_N_k_prob}
is as follows
$$
\sum_{i=1}^{3}  h_i^2(0) \,\mathrm{v}_i \,  w_k^{(i)}(0, t) = {\bf d}_k(t),
$$
where
\begin{align}\label{d_k}
 {\bf d}_k(t) :=&  - \frac{1}{\pi}\int_{\Xi^{(0)}} \dot{{N}}_{k-1}(\xi,t)\, d\xi - \frac{1}{\pi} \sum_{i=1}^{3}\int_{\Xi^{(i)}} \dot{\widetilde{N}}_{k-1}(\xi,t) \, d\xi
\notag
\\
 & - \sum_{i=1}^{3} h^2_i(0) \Big(\mathrm{v}_i  \int_{\ell_0+1}^{\ell_0+2} \Psi^{(i)}_{k}(\xi_i,t) \,  \chi^\prime_{\ell_0}(\xi_i) \, d\xi_i -  a^{(i)}_{ii} \, \int_{\ell_0+1}^{\ell_0+2} \big(\Psi^{(i)}_{k}(\xi_i,t)\big)' \,  \chi^\prime_{\ell_0}(\xi_i) \, d\xi_i \Big),
\end{align}
and the function $\Psi^{(i)}_{k}$ is determined in \eqref{Psi_k}.

Using  \eqref{model_solution} similar as in \S~\ref{sub_limit_problem}, we get a unique smooth solution to the problem
\begin{equation}\label{prob_w_k}
 \left\{\begin{array}{rcll}
h^2_i(x_i)\, \partial_t{w}^{(i)}_k + \partial_{x_i}\big( v_i^{(i)}(x_i)\,  h^2_i(x_i)\, w^{(i)}_k\big) &=& {\bf g}_k^{(i)}(x_i,t),& (x_i, t) \in I_i \times (0, T),  \ \ i \in \{1,2,3\},
 \\[3mm]
\sum_{i=1}^{3}  \mathrm{v}_i \, h_i^2(0) \,  w_k^{(i)}(0, t) & = & {\bf d}_k(t), & t \in (0, T),
 \\[2mm]
    w_k^{(i)}(\ell_i,  t) & = & 0, & t \in [0, T],   \ \ i \in \{1,2\},
\\[2mm]
    w_k^{(i)}(x_i,  0) & = & 0, & x_i  \in [0, \ell_i],  \ \ i \in \{1,2,3\},
  \end{array}\right.
\end{equation}
where
\begin{align*}
{\bf g}_k^{(i)}(x_i,t) := & \ a_{ii}^{(i)} \partial_{x_i} \big(h^2_i(x_i)\, \partial_{x_i} w_{k-1}^{(i)}(x_i,t)\big)
-  v_i^{(i)}(x_i) \, \partial_{x_i}\big(h^2_i(x_i)\big) \,  \widehat{u}^{(i)}_k(x_i, t)
\\
&+ a_{ii}^{(i)} \partial_{x_i}\big(h^2_i(x_i)\big) \,   \partial_{x_i}\widehat{u}^{(i)}_{k-1}(x_i, t) -  2 h_i(x_i)\, \eta_k^{(i)}(x_i) \,  \widehat{\varphi}^{(i)}(x_i, t),
\end{align*}
and the function $\widehat{u}^{(i)}_k$ is determined in \eqref{hat_u_k}.

Thus, the solvability condition  for the problem \eqref{tilda_N_k_prob} is satisfied and there exists a unique solution $N_k$ to the problem \eqref{N_k_prob}.  In addition, the right-hand sides of the problem \eqref{tilda_N_k_prob} are equal to zero at $t=0.$ As a result,
$ N_k\big|_{t=0} \equiv \widetilde{N}_k\big|_{t=0} \equiv 0.$ Similarly as at the end of  \S~\ref{sub_limit_problem}, we deduce that
$\partial_t N_k\big|_{t=0} \equiv 0.$
 Also the solvability condition \eqref{lim_2} holds for the problem  \eqref{eq_3}, \eqref{bc_3} and \eqref{uniq_1},
and therefore, there  exists a unique solution $u_{k+1}^{(i)}$; moreover,  as is easy to see, it vanishes at $t=0.$

By the same argumentations as for \eqref{rem_exp-decrease+0} and \eqref{rem_exp-decrease+1},   we can state that  the solution to the problem~\eqref{N_k_prob}  has  the following  asymptotics uniform with respect to $t\in [0, T]$:
\begin{equation}\label{rem_exp-decrease}
N_k(\xi,t) = w^{(i)}_{k}(0,t) + \Psi^{(i)}_{k}(\xi_i,t)  +  \mathcal{ O}(\exp(-\beta_0\xi_i))
\quad \mbox{as} \ \ \xi_i\to+\infty,  \ \  \xi  \in \Xi^{(i)},  \quad i=\{1,2,3\} \quad (\beta_0 >0).
\end{equation}

\subsubsection{Limit problem $($the second case: $v^{(1)}_1 < 0,\  v^{(2)}_2 >  0, \  v^{(3)}_3 > 0 )$}\label{sub_limit_problem+}

As follows from \cite{Bar_Le Roux_1979}, the boundary condition is not necessarily satisfied for $w_0^{(2)}$ and $w_0^{(3)}$ for $x_2=\ell_2$ and $x_3 =\ell_3, $ accordingly, even if we consider a weak solution that satisfies the so-called ``entropy-flow boundary pairs'' condition (see also \cite{Martin, Otto}). Therefore, we consider in this case the following limit problem:
\begin{equation}\label{limit_prob_2_case}
 \left\{\begin{array}{rcll}
h^2_i(x_i)\, \partial_t{w}^{(i)}_0 + \partial_{x_i}\big( v_i^{(i)}(x_i)\,  h^2_i(x_i)\, w^{(i)}_0 \big) &=& - 2 h_i(x_i)\,  \widehat{\varphi}^{(i)}(x_i, t),& (x_i, t) \in I_i \times (0, T),  \ \ i \in \{1,2,3\},
 \\[3mm]
\sum_{i=1}^{3}  \mathrm{v}_i \, h_i^2(0) \,  w_0^{(i)}(0, t) & = & \breve{\varphi}^{(0)}(t),& t \in (0, T),
    \\[2mm]
    w_0^{(1)}(\ell_1,  t) & = & q_1(t), & t \in [0, T],
    \\[2mm]
    w_0^{(i)}(x_i,  0) & = & 0, & x_i  \in [0, \ell_i],  \ \ i \in \{1,2,3\}.
  \end{array}\right.
\end{equation}

The function $w_0^{(1)}$ is the smooth solution to the problem \eqref{limit_prob_1_2}
 and it is defined  by the formula \eqref{model_solution_w_0}. Compared to the first case, the dynamics at
the  vertex is  not uniquely determined by the mass conservation condition \eqref{cong_cond}.  We propose an additional relation between the incoming solution $w_0^{(1)}$ and the outgoing solution $w_0^{(2)}$ at the vertex  of the graph $\mathcal{I},$ namely
\begin{equation}\label{cont_condition}
w_0^{(2)}(x_2,t)\big|_{x_2 =0} =  w_0^{(1)}(0,t), \quad t \in [0,T]
\end{equation}
 (our choice is argued and discussed for general graphs  in the Conclusions). In virtue of the additional assumptions in Remark~\ref{differ},  there exists a smooth solution to the problem
 \begin{equation}\label{limit_prob_2+}
 \left\{\begin{array}{rcll}
h^2_2(x_2)\, \partial_t{w}^{(2)}_0(x_2, t) + \partial_{x_2}\big( v_i^{(2)}(x_2)\,  h^2_2(x_2)\, w^{(2)}_0(x_2, t) \big) &=& - 2 h_2(x_2)\,  \widehat{\varphi}^{(2)}(x_2, t),& (x_2, t) \in I_2 \times (0, T),
   \\[2mm]
    w_0^{(2)}(0,  t) & = & w_0^{(1)}(0,t), & t \in [0, T],
\\[2mm]
    w_0^{(2)}(x_2,  0) & = & 0, & x_2  \in [0, \ell_2],
  \end{array}\right.
\end{equation}
and it is determined with help of \eqref{model_solution}. Further,  as in the first case, we determine the function $w_0^{(3)}$ (see the problem \eqref{limit_prob_w_3}).

For the following terms, we repeat all  calculations and argumentations from \S~\ref{par_222} and \S~\ref{par_223}, with an amendment to determine the coefficient $w_k^{(2)}.$ First we determine  $w_k^{(1)}$ as a smooth solution to the corresponding differential equation (see \eqref{prob_w_k}) in $I_1\times (0, T)$ with the conditions $w_k^{(1)}\big|_{x_1=\ell_1}=0$ and $w_k^{(1)}\big|_{t=0}=0;$ then the  coefficient $w_k^{(2)}$ but with the conditions $w_k^{(2)}\big|_{x_2=0}= w_k^{(1)}(0,t)$ and $w_k^{(2)}\big|_{t=0}=0;$ and at the end we determine the  coefficient $w_k^{(3)}$ as in the first case.

%
\subsection{Boundary-layer part of the asymptotics at the base $\Upsilon_{\varepsilon}^{(3)} (\ell_3)$}\label{subsec_Bound_layer}

In Subsection~\ref{regular_asymptotic}, we considered the
regular asymptotics taking into account the inhomogeneity on the lateral surfaces  of the thin cylinders $\{\Omega^{(i)}_\varepsilon\}_{i=1}^3$ and the boundary conditions on the bases $\Upsilon_{\varepsilon}^{(1)} (\ell_1)$ and
$\Upsilon_{\varepsilon}^{(2)} (\ell_2)$ for the first case:  $v^{(1)}_1 < 0,$ $ v^{(2)}_2 < 0,$ $ v^{(3)}_3 > 0 $.  In what follows, we construct the boundary layer part of the asymptotics compensating the residuals of the regular one at the base $\Upsilon_{\varepsilon}^{(3)} (\ell_3)$ of the thin cylinder $\Omega^{(3)}_\varepsilon$.

We seek it in the form
\begin{equation}\label{prim+}
\sum\limits_{k=0}^{+\infty}\varepsilon^{k} \, \Pi_k^{(3)}\left(\frac{{x}_1}{\varepsilon}, \frac{{x}_2}{\varepsilon}, \frac{\ell_3-x_3}{\varepsilon}, t\right)
\end{equation}
in a neighborhood of $\Upsilon_{\varepsilon}^{(3)} (\ell_3).$

We additionally assume that component $v_3^{(3)}$ of the vector-valued function $\overrightarrow{V_\varepsilon}^{(3)}$ are independent of the variable $x_3$ in a neighborhood of $\Upsilon_{\varepsilon}^{(3)}(\ell_3),$ i.e.,
$$
\overrightarrow{V_\varepsilon}^{(3)} = \big( 0, 0, v_3^{(3)}(\ell_3) \big)
$$
in a neighborhood of $\Upsilon_{\varepsilon}^{(3)} (\ell_3).$ This is a technical assumption. In the general case, the function  $v_3^{(3)}$ need to be expanded in terms of  Taylor series in a neighborhood of the point $x_3=\ell_3.$

Substituting the series \eqref{prim+}  into \eqref{probl}  and collecting  coefficients at the same powers of $\varepsilon$,
we get the following mixed boundary value problems:
 \begin{equation}\label{prim+probl+0}
 \left\{\begin{array}{rcll}
    \mathrm{div}_{\bar{\xi}_3} \big( \tilde{\mathbb{D}}^{(3)}_{\bar{\xi}_3} \nabla_{\bar{\xi}_3}\Pi_0^{(3)}\big)
  + a_{33}^{(3)} \, \partial^2_{\xi_3 \xi_3}\Pi_0^{(3)} + v_3^{(3)}(\ell_3) \partial_{\xi_3}\Pi_0^{(3)}
  & =    & 0,
   & \quad \xi\in \mathfrak{C}_+^{(3)},
   \\[2mm]
  \partial_{\nu_{\overline{\xi}_3}} \Pi_0^{(3)}(\xi,t) & =
   & 0,
   & \quad \xi\in \mathfrak{C}_+^{(3)},
   \\[2mm]
  \Pi_0^{(3)}(\overline{\xi}_3,0,t) & =
   & \Phi_0(t),
   & \quad \overline{\xi}_3\in\Upsilon_3(\ell_3),
   \\[2mm]
  \Pi_0^{(3)}(\xi,t) & \to
   & 0,
   & \quad \xi_3\to+\infty,
 \end{array}\right.
\end{equation}
where $\Phi_0(t) = q_{3}(t) - w_{0}^{(3)}(\ell_3,t),$ \
the matrix $\tilde{\mathbb{D}}^{(3)}$ is determined in~\eqref{mat-2D},  \
$\xi=(\xi_1, \xi_2, \xi_3),$ $\xi_3 = \frac{\ell_3-x_3}{\varepsilon},$ $\overline{\xi}_3 =(\xi_1, \xi_2) = \frac{\overline{x}_3}{\varepsilon},$
$$
\mathfrak{C}_+^{(3)}:=\big\{\xi \colon \  \overline{\xi}_3\in\Upsilon_3(\ell_3), \quad \xi_3\in(0,+\infty)\big\};
$$
and
\begin{equation}\label{prim+probl+k}
 \left\{\begin{array}{rcll}
    \mathrm{div}_{\bar{\xi}_3} \big( \tilde{\mathbb{D}}^{(3)}(\bar{\xi}_3) \nabla_{\bar{\xi}_3}\Pi_k^{(3)}\big)
  + a_{33}^{(3)} \, \partial^2_{\xi_3 \xi_3}\Pi_k^{(3)} + v_3^{(3)}(\ell_3) \partial_{\xi_3}\Pi_k^{(3)}
  & =    & \partial_t \Pi_{k-1}^{(3)} ,
   & \quad \xi\in \mathfrak{C}_+^{(3)},
   \\[2mm]
   \partial_{\nu_{\overline{\xi}_3}} \Pi_k^{(3)}(\xi) & =
   & 0,
   & \quad \xi\in \mathfrak{C}_+^{(3)},
   \\[2mm]
  \Pi_k^{(3)}(\overline{\xi}_3,0,t) & =
   & \Phi_k(t),
   & \quad \overline{\xi}_3\in\Upsilon_3(\ell_3),
   \\[2mm]
  \Pi_k^{(3)}(\xi,t) & \to
   & 0,
   & \quad \xi_3\to+\infty,
 \end{array}\right.
\end{equation}
where $\Phi_k(t) =  - \, w_{k}^{(3)}(\ell_3,t),  \  k\in \Bbb N.$

Using  the  Fourier method, it is easy to find
\begin{equation}\label{Pi_0}
\Pi_0^{(3)}(\xi_3,t) =
\Phi_0(t) \, e^{- \lambda \, \xi_3} , \ \ \text{where} \ \ \lambda := \frac{v_3^{(3)}(\ell_3)}{a_{33}^{(3)}}.
\end{equation}
Since  $v_3^{(3)}(\ell_3) > 0$ and $a_{33}^{(3)} >0,$ the solution $\Pi_0^{(3)}$ tends exponentially to zero as $\xi_3\to+\infty$ uniformly with respect to $t\in [0, T].$ Other coefficients are also independent of the variables $\bar{\xi}_3$ and they can be determined step-by-step from the following sequence of  problems
\begin{equation}\label{prim+probl+k+F}
 \left\{\begin{array}{rcll}
     \partial^2_{\xi_3 \xi_3}\Pi_k^{(3)}(\xi_3,t) + \lambda \, \partial_{\xi_3}\Pi_k^{(3)}(\xi_3,t)
  & =    & \frac{1}{a_{33}^{(3)}} \, \partial_t \Pi_{k-1}^{(3)}(\xi_3,t) ,
   & \quad \xi_3 \in (0, +\infty),
   \\[2mm]
  \Pi_k^{(3)}(0,t) & =   & \Phi_k(t),
   &
   \\[2mm]
  \Pi_k^{(3)}(\xi_3,t) & \to
   & 0,
   & \quad \xi_3\to+\infty; \quad k \in \Bbb N.
 \end{array}\right.
\end{equation}
For example, the problem
\begin{equation}\label{prim+probl+1+F}
 \left\{\begin{array}{rcll}
     \partial^2_{\xi_3 \xi_3}\Pi_1^{(3)}(\xi_3,t) + \lambda \, \partial_{\xi_3}\Pi_1^{(3)}(\xi_3,t)
  & =    & -\frac{1}{a_{33}^{(3)}} \, \dot{\Phi}_0(t)\,  e^{-\lambda \xi_3},
   & \quad \xi_3 \in (0, +\infty),
   \\[2mm]
  \Pi_1^{(3)}(0,t) & =   & \Phi_1(t),
   &
   \\[2mm]
  \Pi_1^{(3)}(\xi_3,t) & \to
   & 0,
   & \quad \xi_3\to+\infty,
 \end{array}\right.
\end{equation}
has the solution
$$
\Pi_1^{(3)}(\xi_3,t) = \Big(\Phi_1(t) +   \tfrac{1}{\lambda \, a_{33}^{(3)}} \, \dot{\Phi}_0(t)\, \xi_3 \Big) \,  e^{-\lambda \xi_3}.
$$

Thus, we can uniquely define all coefficients of the series \eqref{prim+} and for each $k \in \Bbb N_0$
\begin{equation}\label{exp_decay}
  \Pi_k^{(3)}(\xi_3,t) = \mathcal{O}\big(e^{-\lambda \xi_3}\big) \quad \text{as} \quad \xi_3 \to +\infty
\end{equation}
uniformly with respect to $t\in [0, T].$ Obviously, one obtains for  each $k \in \Bbb N_0$
\begin{equation}\label{in_cond}
  \Pi_k^{(3)}\big|_{t=0} \equiv 0 .
\end{equation}

In the second case, we similarly construct boundary asymptotics
\begin{equation}\label{prim++}
\sum\limits_{k=0}^{+\infty}\varepsilon^{k} \, \Pi_k^{(2)}\left(\frac{\ell_2-x_2}{\varepsilon},  t\right)
\end{equation}
in a neighborhood of $\Upsilon_{\varepsilon}^{(2)} (\ell_2)$ with the same properties.

\section{Formal asymptotic solution to the problem \eqref{probl}}\label{Sec:justification}

With the results from  the previous section  we  can successively determine all coefficients of the series (\ref{regul}), (\ref{inner_part}) and (\ref{prim+}) determining the complete solution. Precisely, we construct the following series:

\begin{equation}\label{asymp_expansion}
    \mathcal{U}(x, t; \varepsilon) :=
        \sum\limits_{k = 0}^{+\infty} \varepsilon^k
        \Big(
            \widehat{u}_k (x, t; \,\varepsilon)
          + \widehat{N}_k (x,t; \,\varepsilon)
          + \widehat{\Pi}_k (x, t;\,\varepsilon)
        \Big), \quad
    x\in\Omega_\varepsilon.
\end{equation}
where
\begin{gather*}
    \widehat{u}_k (x, t; \,\varepsilon)
     := \sum\limits_{i=1}^3
        \chi_{\ell_0}^{(i)} \left(\frac{x_i}{\varepsilon^\gamma}\right)
        \Big(
            w_k^{(i)} (x_i, t)
          + u_k^{(i)} \Big( x_i, \frac{\overline{x}_i}{\varepsilon}, t \Big)
        \Big), \quad
   (u_0 \equiv 0),
\\
    \widehat{N}_k (x,t; \, \varepsilon)
     := \bigg(
            1
          - \sum\limits_{i = 1}^3
            \chi_{\ell_0}^{(i)} \left(\frac{x_i}{\varepsilon^\gamma}\right)
        \bigg)
        N_k \Big( \frac{x}{\varepsilon}, t \Big),
\\
    \widehat{\Pi}_k (x, t; \varepsilon)
     := \chi_\delta^{(3)} (x_3) \,
        \Pi_k^{(3)}
        \Big(\frac{\ell_3 - x_3}{\varepsilon}, t
        \Big), \quad
        k\in\mathbb{N}_0,
\end{gather*}
$\gamma$ is a fixed number from the interval $(\frac23, 1),$
$\chi_\delta^{(3)}, \  \chi_{\ell_ 0}^{(i)}$
are smooth cut-off functions defined by the formulas
\begin{equation}\label{cut-off_functions}
\chi_\delta^{(3)} (x_3) =
    \left\{
    \begin{array}{ll}
        1, & \text{if} \ \ x_3 \ge \ell_3 -  \delta,
    \\
        0, & \text{if} \ \ x_3 \le \ell_3 - 2\delta,
    \end{array}
    \right.
\qquad
\chi_{\ell_0}^{(i)} (x_i) =
\left\{\begin{array}{ll}
1, &  \ \ x_i \ge 3 \, \ell_0,
\\
0, &  \ \ x_i \le 2 \, \ell_0,
\end{array}\right.
\quad i \in \{1, 2, 3\},
\end{equation}
and $\delta$ is a sufficiently small fixed positive number such that $\chi_\delta^{(3)}$ vanishes in the support of $\varphi_\varepsilon^{(3)}.$

Here we again consider the first case; in the second case, one should add the asymptotics of the boundary layer near the base of the thin cylinder $\Omega^{(2)} _\varepsilon$ multiplying by the corresponding cut-off function.

For $M\in \Bbb N$, denote by
\begin{equation}\label{aaN}
\mathcal{U}_{M}(x, t;\varepsilon)
 :=  \sum\limits_{k=0}^{M} \varepsilon^{k}
    \Big(
            \widehat{u}_k (x, t; \,\varepsilon)
          + \widehat{N}_k (x,t; \,\varepsilon)
          + \widehat{\Pi}_k (x, t;\,\varepsilon)
        \Big), \quad
    x\in\Omega_\varepsilon
    \end{equation}
 the partial sum of $(\ref{asymp_expansion}).$ Obviously,
 $$
 \mathcal{U}_{M}\big|_{t=0} = 0, \qquad \mathcal{U}_{M}\big|_{x_i=\ell_i} = q_i(t), \quad i\in \{1, 2, 3\}.
 $$

 In $\Omega^{(i)}_{\varepsilon, \gamma} := \Omega^{(i)} _\varepsilon \cap \{x\colon x_i \in [3 \ell_0 \varepsilon^\gamma, \ell_i) \},$ $i\in \{1, 2\},$
 \begin{equation}\label{sum_1}
   \mathcal{U}_{M}(x, t;\varepsilon) = \sum\limits_{k=0}^{M} \varepsilon^{k}
 \Big( w_k^{(i)} (x_i, t) + u_k^{(i)}\Big( x_i, \frac{\overline{x}_i}{\varepsilon}, t \Big)\Big)
 \end{equation}
 and due to the  equations \eqref{eq_1} and \eqref{eq_3} for $\{u_k^{(i)}\}$ this partial sum satisfies the differential equation
 \begin{equation}\label{Res_1}
    \partial_t\,\mathcal{U}_{M} -  \varepsilon\, \mathrm{div}_x \big( \mathbb{D}^{(i)}_\varepsilon(x) \nabla_x \mathcal{U}_{M}\big) +
  \mathrm{div} \big( \overrightarrow{V_\varepsilon}(x) \, \mathcal{U}_{M}\big)
       = \varepsilon^M\, \mathcal{R}^{(i)}_M \quad \text{in} \ \ \Omega^{(i)}_{\varepsilon, \gamma}\times (0,T)
 \end{equation}
where
\begin{multline}\label{Res_1+}
  \mathcal{R}^{(i)}_M(x_i,\bar{\xi}_i, t) =  -
        a_{ii}^{(i)}
        \Big(
            w^{(i)}_{M-1}(x_i,t)
            +
            u^{(i)}_{M-1}(x_i, \bar{\xi}_i,t)
        \Big)^{\prime\prime} + \partial_t{w}^{(i)}_M(x_i,t) + \partial_t {u}^{(i)}_{M}(x_i, \bar{\xi}_i,t)
\\
        +
        \Big(
            v^{(i)}(x_i) \,
            \big[ w^{(i)}_M(x_i,t) + u^{(i)}_{M}(x_i, \bar{\xi}_i,t) \big]
        \Big)^\prime
        +
        \mathrm{div}_{\bar{\xi}_i}
        \big(
            \overline{V}^{(i)}(x_i, \bar{\xi}_i) \,
            \big[ w^{(i)}_{M}(x_i,t) + u^{(i)}_{M}(x_i, \bar{\xi}_i,t) \big]
        \big)
   \\
 - \varepsilon \,  a_{ii}^{(i)}  \Big(
            w^{(i)}_{M}(x_i,t)
            +
            u^{(i)}_{M}(x_i, \bar{\xi}_i,t)
        \Big)^{\prime\prime}.
\end{multline}
It is easy to verify that thanks to our assumptions
\begin{equation}\label{Res_2}
  \sup_{\Omega^{(i)} _{\varepsilon,\gamma}\times (0,T)} |\mathcal{R}^{(i)}_M(x_i,\tfrac{\bar{x}_i}{\varepsilon}, t)| \le C^{(i)}_M,
\end{equation}
where the constant $C^{(i)}_M$ is independent of $\varepsilon.$
\begin{remark}
Hereinafter, all constants in inequalities are independent of the parameter~$\varepsilon.$
\end{remark}

In $\Omega^{(3)}_{\varepsilon, \gamma} := \Omega^{(3)} _\varepsilon \cap \{x\colon x_3 \in [3 \ell_0 \varepsilon^\gamma, \ell_3) \},$
 \begin{equation}\label{sum_2}
 \mathcal{U}_{M}(x, t;\varepsilon) = \sum\limits_{k=0}^{M} \varepsilon^{k}
 \Big( w_k^{(3)} (x_3, t) + u_k^{(3)}\Big( x_3, \frac{\overline{x}_3}{\varepsilon}, t \Big) + \chi_\delta^{(3)} (x_3) \,
        \Pi_k^{(3)}
        \Big( \frac{\ell_3 - x_3}{\varepsilon}, t  \Big)\Big)
 \end{equation}
 and due to the differential equations \eqref{eq_1}--\eqref{eq_3} for $\{u_k^{(3)}\}$
 and the differential equations \eqref{prim+probl+0}, \eqref{prim+probl+k+F} for $\{\Pi_k^{(3)}\}$
  this partial sum satisfies the differential equation
 \begin{multline}\label{Res_3}
    \partial_t\,\mathcal{U}_{M} -  \varepsilon\, \mathrm{div}_x \big( \mathbb{D}^{(3)}_\varepsilon(x) \nabla_x \mathcal{U}_{M}\big) +
  \mathrm{div} \big( \overrightarrow{V_\varepsilon}(x) \, \mathcal{U}_{M}\big)
       = \varepsilon^M\, \Big(\mathcal{R}^{(3)}_M + \chi_\delta^{(3)} (x_3) \, \partial_t \Pi_M^{(3)} \Big)
       \\
   - \varepsilon \, a_{33}^{(3)} \,  \big(\chi_\delta^{(3)}\big)^{\prime\prime}  \sum\limits_{k=0}^{M} \varepsilon^{k} \Pi_k^{(3)}
       + 2 a_{33}^{(3)} \,  \big(\chi_\delta^{(3)}\big)'  \, \sum\limits_{k=0}^{M} \varepsilon^{k} \partial_{\xi_3} \Pi_k^{(3)}
  + v_3^{(3)}(\ell_3)  \,  \big(\chi_\delta^{(3)}\big)'  \, \sum\limits_{k=0}^{M} \varepsilon^{k} \Pi_k^{(3)}
  \end{multline}
in $\Omega^{(3)}_{\varepsilon, \gamma}\times (0, T).$  The supports of summands in the second line of \eqref{Res_3} coincide with $\mathrm{supp}\big(\big(\chi_\delta^{(3)}\big)'\big),$ where
the functions $\{\Pi_k^{(3)}\}_{k=0}^{M}$ exponentially decay as $\varepsilon$ tends to zero. Therefore, the right-hand side of the differential equation \eqref{Res_3} has also the order $\varepsilon^M$ for $\varepsilon$ small enough.

Based on \eqref{bc_1} and \eqref{bc_3},  the partial sums \eqref{sum_1} and \eqref{sum_2} satisfy, respectively, the following relations
\begin{equation}\label{Res_3+}
  \big(-  \varepsilon \,  \mathbb{D}^{(i)}_\varepsilon\nabla_x \mathcal{U}_{M} +  \mathcal{U}_{M} \, \overrightarrow{V}^{(i)}_\varepsilon\big) \cdot \boldsymbol{\nu}_\varepsilon = \varepsilon\, \varphi_\varepsilon^{(i)} + \varepsilon^{M+1} \Phi_M^{(i)} \quad \text{on} \ \
  \Gamma^{(i)}_{\varepsilon, \gamma}\times (0, T), \quad i \in \{1, 2, 3\},
\end{equation}
where  the lateral surfaces $\Gamma^{(i)}_{\varepsilon, \gamma} := \Gamma^{(i)} _\varepsilon \cap \{x\colon x_i \in [3 \ell_0 \varepsilon^\gamma, \ell_i) \},$ $i\in \{1, 2, 3\},$
$$
  \Phi_M^{(i)}(x_i,\bar{\xi}_i, t) = \frac{1}{\sqrt{1 + \varepsilon^2 |h'_i(x_i)|^2}} \Bigg( - \varphi_\varepsilon^{(i)}(x,t) \, \sum_{k=M+1}^{+\infty}\eta^{(i)}_k \varepsilon^{k-M-1}
    - v^{(i)}_{i}(x_i) \, h'_i(x_i)
            \left(w_M^{(i)} (x_i,t) + u_M^{(i)} \left( x_i, \bar{\xi}_i, t \right)\right)
$$
$$
            + a^{(i)}_{ii}\, h'_i(x_i) \left(w_{M-1}^{(i)} (x_i,t) + u_{M-1}^{(i)} \left( x_i, \bar{\xi}_i, t \right)\right)'
                   + \big(\overline{V}^{(i)}(x_i, \bar{\xi}_i) \cdot \bar{\nu}_{\bar{\xi}_i}\big)\left(w_{M}^{(i)} (x_i,t) + u_{M}^{(i)} \left( x_i, \bar{\xi}_i, t \right)\right)
$$
\begin{equation}\label{Res_4}
            + \varepsilon\, a^{(i)}_{ii}\, h'_i(x_i) \left(w_{M}^{(i)} (x_i,t) + u_{M}^{(i)} \left( x_i, \bar{\xi}_i, t \right)\right)'
            \Bigg)
\end{equation}
It is easy to verify  that due to our assumptions there are positive constants $\varepsilon_0$ and $\tilde{C}^{(i)}_M$ such that for all $\varepsilon \in (0, \varepsilon_0)$ we have
\begin{equation}\label{Res_5}
  \sup_{\Gamma^{(i)}_{\varepsilon, \gamma} \times (0,T)} |\Phi_M^{(i)}(x_i,\tfrac{\bar{x}_i}{\varepsilon}, t)| \le \tilde{C}^{(i)}_M;
\end{equation}
in addition, the functions $\{\Phi_M^{(i)}\}_{i=1}^3$ vanish  on circular strips on the lateral surfaces of the thin cylinders near their bases
$\{\Upsilon_{\varepsilon}^{(i)} (\ell_i)\times (0, T)\}_{i=1}^3$ since the functions $\{h'_i\},$ $\{\overline{V}^{(i)}\},$  and  $\{\varphi_\varepsilon^{(i)}\}$ vanish there.

In the  neighborhood $\Omega^{(0)}_{\varepsilon, \gamma}:= \Omega_\varepsilon \cap \{x\colon x_i \in [0, 2\ell_0 \varepsilon^\gamma], \ \ i=1, 2, 3\}$ of the node $\Omega^{(0)}_{\varepsilon},$
$$
 \mathcal{U}_{M}(x, t;\varepsilon) = \sum\limits_{k=0}^{M} \varepsilon^{k} \, N_k \Big( \frac{x}{\varepsilon}, t \Big)
$$
and due to the problems \eqref{N_0_prob} and \eqref{N_k_prob}  this partial sum satisfies the differential equation
\begin{equation}\label{Res_6}
    \partial_t\,\mathcal{U}_{M} -  \varepsilon\, \mathrm{div}_x \big( \mathbb{D}^{(0)}_\varepsilon(x) \nabla_x \mathcal{U}_{M}\big) +
  \mathrm{div} \big( \overrightarrow{V_\varepsilon}(x) \, \mathcal{U}_{M}\big)
       = \varepsilon^M\,  \partial_t N_M \Big( \frac{x}{\varepsilon}, t \Big) \quad \text{in} \ \ \Omega^{(0)}_{\varepsilon, \gamma}\times (0,T);
 \end{equation}
 and the boundary condition
\begin{equation}\label{Res_7}
\big(-  \varepsilon \,  \mathbb{D}^{(0)} _\varepsilon\nabla_x \mathcal{U}_{M} +  \mathcal{U}_{M} \, \overrightarrow{V}_\varepsilon\big) \cdot \boldsymbol{\nu}_\varepsilon   =   \varphi^{(0)}_\varepsilon \quad
   \text{on} \ \Big(\partial\Gamma^{(0)}_{\varepsilon, \gamma} \setminus \Big\{
 \bigcup_{i=1}^3 \overline{\Upsilon_\varepsilon^{(i)} (2\ell_0 \varepsilon^\gamma)}
\Big\} \Big) \times (0, T).
 \end{equation}
It should be noted here that we can continue  the function $\varphi^{(0)}_\varepsilon$ by zero to  the remaining  lateral part, and  the functions $\{\varphi^{(i)}_\varepsilon\}_{i=1}^3$ vanish
on $\Gamma^{(i)} _\varepsilon \cap \{x\colon x_i \in [ \varepsilon \ell_0, 3 \ell_0 \varepsilon^\gamma]\},$ $i\in \{1, 2, 3\},$
respectively.

Now it remains to calculate the residuals left by the partial sum \eqref{aaN} in the differential equations in
$\Omega^{(i)} _\varepsilon \cap \{x\colon x_i \in [ 2 \ell_0 \varepsilon^\gamma, 3 \ell_0 \varepsilon^\gamma]\},$ $i\in \{1, 2, 3\},$
and in the boundary conditions on the corresponding lateral surfaces. There, owing to   Remark~\ref{r_2_1},  the partial sum reads as follows
$$
 \mathcal{U}_{M}(x, t;\varepsilon) = \sum\limits_{k=0}^{M} \varepsilon^{k} \, \bigg(
 \chi_{\ell_0}^{(i)} \left(\frac{x_i}{\varepsilon^\gamma}\right) \,  w_k^{(i)} (x_i, t) +
 \Big(1 -    \chi_{\ell_0}^{(i)} \left(\frac{x_i}{\varepsilon^\gamma}\right)\Big)\,  N_k \Big( \frac{x}{\varepsilon}, t \Big) \bigg).
$$

Since $h'_i, \, \varphi^{(i)}_\varepsilon,\,  \overline{V}^{(i)}$ vanish for $x_i \in [ 2 \ell_0 \varepsilon^\gamma, 3 \ell_0 \varepsilon^\gamma],$
\begin{equation}\label{Res_8}
\big(-  \varepsilon \,  \mathbb{D}^{(i)} _\varepsilon\nabla_x \mathcal{U}_{M} +  \mathcal{U}_{M} \, \overrightarrow{V}_\varepsilon\big) \cdot \boldsymbol{\nu}_\varepsilon   =  0
 \end{equation}
on the corresponding lateral surface of the domain $\Omega^{(i)} _\varepsilon \cap \{x\colon x_i \in [ 2 \ell_0 \varepsilon^\gamma, 3 \ell_0 \varepsilon^\gamma]\}.$ In this domain, based on \eqref{Res_1+} and \eqref{Res_6}, we have
\begin{multline}\label{Res_9}
   \partial_t\,\mathcal{U}_{M} -  \varepsilon\, \mathrm{div} \big( \mathbb{D}^{(0)}_\varepsilon(x) \nabla \mathcal{U}_{M}\big) +
  \mathrm{div} \big( \overrightarrow{V_\varepsilon}(x) \, \mathcal{U}_{M}\big)
  \\
       =
       \varepsilon^M\,  \bigg\{\chi_{\ell_0}^{(i)} \left(\frac{x_i}{\varepsilon^\gamma}\right)
      \bigg(  -  a_{ii}^{(i)} \Big( w^{(i)}_{M-1}(x_i,t)\Big)^{\prime\prime}
  +  \mathrm{v}_i \Big(w^{(i)}_M(x_i,t) \Big)^\prime
        - \varepsilon  a_{ii}^{(i)}  \Big(w^{(i)}_{M}(x_i,t)\Big)^{\prime\prime}\bigg)
      +   \Big(1 -    \chi_{\ell_0}^{(i)} \left(\frac{x_i}{\varepsilon^\gamma}\right)\Big) \partial_t N_M \bigg\}
      \\
      - a_{ii} \left(\chi_{\ell_0}^{(i)}\right)'' \sum\limits_{k=0}^{M} \varepsilon^{k+1 -2\gamma}    \Big(w_k^{(i)} (x_i, t) - N_k\Big)
      - 2 a_{ii} \, \left(\chi_{\ell_0}^{(i)}\right)' \sum\limits_{k=0}^{M} \varepsilon^{k+1 -\gamma}  \Big(\partial_{x_i} w_k^{(i)} (x_i, t) - \varepsilon^{-1} \partial_{\xi_i} N_k\Big)
      \\
      + \mathrm{v}_i \, \left(\chi_{\ell_0}^{(i)}\right)' \sum\limits_{k=0}^{M} \varepsilon^{k -\gamma}   \Big(w_k^{(i)} (x_i, t) - N_k\Big).
      \end{multline}
Summands in the first line of the right-hand side of \eqref{Res_9} are of the order $\varepsilon^M.$ The other summands are localized in
the support of  $\big(\chi_{\ell_0}^{(i)}\big)'$, i.e., in  $ \Omega_\varepsilon^{(i)} \cap \big\{ x: \  x_i\in  [2\ell_0\varepsilon^\gamma, 3\ell_0\varepsilon^\gamma] \big\}. $ Therefore, using  the Taylor formula  for the functions $\{w_{k}^{(i)}\}_{k=0}^M$ at the point $x_i=0$
and the formula \eqref{new-solution_k},  summands in the last two lines of \eqref{Res_9} can be rewritten as follows
\begin{multline*}
 a_{ii} \left(\chi_{\ell_0}^{(i)}\right)'' \sum\limits_{k=0}^{M} \varepsilon^{k+1 -2\gamma}    \widetilde{N}_k\Big(\frac{x}{\varepsilon},t\Big) + \mathcal{O}(\varepsilon^M)  + 2 a_{ii} \, \left(\chi_{\ell_0}^{(i)}\right)' \sum\limits_{k=0}^{M} \varepsilon^{k -\gamma} \,  \partial_{\xi_i}\widetilde{N}_k(\xi,t)\big|_{\xi=\frac{x}{\varepsilon}} + \mathcal{O}(\varepsilon^{M-\gamma})
 \\
 - \mathrm{v}_i \, \left(\chi_{\ell_0}^{(i)}\right)' \sum\limits_{k=0}^{M} \varepsilon^{k -\gamma}   \widetilde{N}_k\Big(\frac{x}{\varepsilon,t}\Big) + \mathcal{O}(\varepsilon^M) \quad \text{as} \ \ \varepsilon \to 0.
 \end{multline*}
Taking into account  \eqref{rem_exp-decrease}, the  maximum of $|\widetilde{N}_k|$ and $|\partial_{\xi_i}\widetilde{N}_k|$ over  $ \big(\Omega_\varepsilon^{(i)} \cap \big\{ x: \  x_i\in  [2\ell_0\varepsilon^\gamma, 3\ell_0\varepsilon^\gamma] \big\} \big) \times [0, T]$ are of the order $\exp\big(-\beta_0 2 \ell_0 \, \varepsilon^{\gamma -1}\big),$ i.e.,
these terms exponentially decrease as the parameter $\varepsilon$ tends to zero.  Thus,  the right-hand side of \eqref{Res_9} has the order  $\varepsilon^{M - \gamma}.$

Based on the calculations in this section, the following statement holds.

\begin{proposition}\label{Prop-3-1} For any  $M\in \Bbb N, $ there is a positive number $\varepsilon_0$ such that for all $\varepsilon\in (0, \varepsilon_0)$  the difference between the partial sum \eqref{aaN} and the solution to the problem \eqref{probl} satisfies the following relations:
\begin{equation}\label{formal_solution}
\left\{
\begin{array}{rcll}
\partial_t(\mathcal{U}_{M} - u_\varepsilon)   -  \varepsilon\, \mathrm{div}_x \big( \mathbb{D}^{(i)}_\varepsilon \nabla_x(\mathcal{U}_{M} - u_\varepsilon)\big) +
  \mathrm{div}_x \big( \overrightarrow{V_\varepsilon}^{(i)} \, (\mathcal{U}_{M} - u_\varepsilon)\big)
  & = & \varepsilon^{M - \gamma} \, \mathcal{R}^{(i)}_M, & \text{in} \ \Omega_\varepsilon^{(i)} \times (0, T),
\\[2mm]
\big(-  \varepsilon \,  \mathbb{D}^{(i)}_\varepsilon\nabla_x(\mathcal{U}_{M} - u_\varepsilon) +  (\mathcal{U}_{M} - u_\varepsilon) \, \overrightarrow{V}^{(i)}_\varepsilon\big) \cdot \boldsymbol{\nu}_\varepsilon   &= &\varepsilon^{M+1} \Phi_M^{(i)} &   \text{on} \  \Gamma_\varepsilon^{(i)}\times (0, T),
\\[2mm]
 (\mathcal{U}_{M} - u_\varepsilon)\big|_{x_i= \ell_i}
 & = & 0, &  \text{on} \ \Upsilon_{\varepsilon}^{(i)} (\ell_i)\times (0, T) ,
\\
& & & i\in\{1,2,3\},
\\[2mm]
\partial_t(\mathcal{U}_{M} - u_\varepsilon) -  \varepsilon\, \mathrm{div}_x \big( \mathbb{D}^{(0)}_\varepsilon \nabla_x(\mathcal{U}_{M} - u_\varepsilon)\big) +
   \overrightarrow{V_\varepsilon}^{(0)} \cdot \nabla_x(\mathcal{U}_{M} - u_\varepsilon) & = & \varepsilon^M\,  \mathcal{R}^{(0)}_M , &
    \text{in} \ \Omega_\varepsilon^{(0)}\times (0, T),
\\[2mm]
 \big(\mathbb{D}^{(0)}_\varepsilon\nabla_x(\mathcal{U}_{M} - u_\varepsilon)\big) \cdot \boldsymbol{\nu}_\varepsilon  & = & 0  &
    \text{on} \ \Gamma_\varepsilon^{(0)}\times (0, T),
\\[2mm]
  (\mathcal{U}_{M} - u_\varepsilon)\big|_{t=0}
 & = & 0, & \text{on} \ \Omega_{\varepsilon},
\end{array}\right.
\end{equation}
where $\gamma$ is a fixed number from the interval $(\frac23, 1),$ $\mathcal{R}^{(0)}_M := \partial_t N_M,$
\begin{equation}\label{Res_10}
  \sup_{\Omega^{(i)} _{\varepsilon}\times (0,T)} |\mathcal{R}^{(i)}_M(x_i,\tfrac{\bar{x}_i}{\varepsilon}, t)| \le C^{(i)}_M, \quad i\in\{0,1,2,3\},
\end{equation}
\begin{equation}\label{Res_11}
  \sup_{\Gamma^{(i)}_{\varepsilon} \times (0,T)} |\Phi_M^{(i)}(x_i,\tfrac{\bar{x}_i}{\varepsilon}, t)| \le \tilde{C}^{(i)}_M, \quad i\in\{1,2,3\},
\end{equation}
and the support of $\Phi_M^{(i)}$  with respect the variable $x_i$ uniformly in $t\in [0, T]$ lies in $(\varepsilon \ell_0, \ell_i)$ $(i\in \{1,2,3\}).$
\end{proposition}

Proposition~\ref{Prop-3-1} means that the series \eqref{asymp_expansion} is the formal asymptotic solution to the problem \eqref{probl}.
To prove that this series is the asymptotic expansion of the solution to the  problem \eqref{probl}, we need a priori estimates, which are very important not only in asymptotic analysis but also in numerical analysis (see e.g. \cite{Stynes-2018}) and the theory of boundary value problems.

%
\section{A priori estimates. The main results}\label{A priori estimates}
\begin{lemma}[the maximum principle]\label{the maximum principle}
  Let  the assumptions ${\bf A1}$ and ${\bf A2}$  made in Section~\ref{Sec:Statement}  be satisfied and
    $\psi_\varepsilon$  be a classical solution to the problem
   \begin{equation}\label{probl+}
\left\{\begin{array}{rcll}
\partial_t\psi_\varepsilon   -  \varepsilon\, \mathrm{div} \big( \mathbb{D}^{(i)}_\varepsilon \nabla \psi_\varepsilon\big) +
  \mathrm{div} \big( \overrightarrow{V_\varepsilon}^{(i)} \, \psi_\varepsilon\big)
  & = & F_i, & \text{in} \ \Omega_\varepsilon^{(i)} \times (0, T), \ \   i\in\{1,2,3\},
\\[2mm]
\big(-  \varepsilon \,  \mathbb{D}^{(i)}_\varepsilon\nabla \psi_\varepsilon +  \psi_\varepsilon \, \overrightarrow{V}^{(i)}_\varepsilon\big) \cdot \boldsymbol{\nu}_\varepsilon   &= & \varepsilon
\, \Phi_i &   \text{on} \  \Gamma_\varepsilon^{(i)}\times (0, T), \ \   i\in\{1,2,3\},
\\[2mm]
\partial_t\psi_\varepsilon -  \varepsilon\, \mathrm{div} \big( \mathbb{D}^{(0)}_\varepsilon \nabla \psi_\varepsilon\big) +
  \overrightarrow{V_\varepsilon}^{(0)} \cdot \nabla\psi_\varepsilon  & = & F_0, &
    \text{in} \ \Omega_\varepsilon^{(0)}\times (0, T),
\\[2mm]
 \big(\mathbb{D}^{(0)}_\varepsilon\nabla \psi_\varepsilon\big) \cdot \boldsymbol{\nu}_\varepsilon  & = & 0  &
    \text{on} \ \Gamma_\varepsilon^{(0)}\times (0, T),
\\[2mm]
 \psi_\varepsilon \big|_{x_i= \ell_i}
 & = & 0, &  \text{on} \ \Upsilon_{\varepsilon}^{(i)} (\ell_i)\times (0, T) , \ \  i\in\{1,2,3\},
\\[2mm]
 \psi_\varepsilon \big|_{t=0}
 & = & 0, & \text{on} \ \Omega_{\varepsilon},
\end{array}\right.
\end{equation}
where the given functions $\{F_i\}_{i=0}^3, \, \{\Phi_i\}_{i=1}^3$ are continuous
and bounded in their domains of definition, abd the support of $\Phi_i$  with respect the variable $x_i$ lies in  $(\varepsilon \ell_0, \ell_i)$ $(i\in \{1,2,3\})$ uniformly in $t\in [0, T]$ .

Then
\begin{equation}\label{max_1}
  \max_{\overline{\Omega_\varepsilon}\times [0, T]} |\psi_\varepsilon|  \le C_0(T) \bigg(\sum_{i=0}^{3}\max_{\overline{\Omega^{(i)}_\varepsilon}\times [0, T]} |F_i| +  \sum_{i=1}^{3}\max_{\overline{\Gamma^{(i)}_\varepsilon}\times [0, T]} |\Phi_i| \bigg),
\end{equation}
where the constant $C_0(T)$ does not depend on  $\varepsilon, \, \psi_\varepsilon, \, \{F_i\}, $ and  $\{\Phi_i\}.$
\end{lemma}

The proof is conducted in the framework of \cite[\S 1]{Il_Ole_Kal_1962} and \cite[Ch. 1, \S 2]{Lad_Sol_Ura_1968}, where  maximal principles for solutions to the first and second initial-boundary value problems were proved;  its novelty lies in finding out how the constant in this estimate depends on the small parameter $\varepsilon$  and choosing a special function that helps to do this.

\begin{proof}
{\bf 1.}  First, let's introduce a new function $y_\varepsilon$ with the help of the substitution $\psi_\varepsilon = y_\varepsilon \, e^{\lambda t},$ where the  positive constant
\begin{equation}\label{e_7}
  \lambda = 1+ \max_{i\in \{1, 2, 3\}} \Big( \max_{[0, \ell_i]}
 \big|\frac{d v_i^{(i)}(x_i)}{d x_i}\big| + \max_{\overline{\Omega^{(i)}_\varepsilon}} \big| \mathrm{div}_{\bar{\xi}_i}  \overline{V}^{(i)}(x_i, \bar{\xi}_i)\big|\Big|_{\bar{\xi}_i=\frac{\bar{x}_i}{\varepsilon}}
    \Big).
\end{equation}
Then the function $y_\varepsilon$ satisfies equations
 \begin{equation}\label{e_1}
  \partial_t y_\varepsilon - \varepsilon \sum_{k,l=1}^{3} a_{kl}^{(0)}  \frac{\partial^2 y_\varepsilon}{\partial x_k \partial x_l}
-   \sum_{k,l=1}^{3} \frac{\partial a_{kl}^{(0)}(\xi) }{\partial \xi_k}\Big|_{\xi=\frac{x}{\varepsilon}}  \frac{\partial y_\varepsilon}{\partial x_l}
+ \overrightarrow{V_\varepsilon}^{(0)} \cdot \nabla_x y_\varepsilon + \lambda y_\varepsilon = e^{-\lambda t} F_0 \quad \text{in} \ \ \Omega_\varepsilon^{(0)}\times (0, T),
\end{equation}
\begin{equation}\label{e_2}
\big(\mathbb{D}^{(0)}_\varepsilon\nabla y_\varepsilon\big) \cdot \boldsymbol{\nu}_\varepsilon   =  0  \quad
    \text{on} \ \Gamma_\varepsilon^{(0)}\times (0, T),
\end{equation}
and (to simplify the notation, we confine ourselves to the equations in the first thin cylinder)
\begin{multline}\label{e_3}
  \partial_t y_\varepsilon - \varepsilon\,  a_{11}^{(1)} \frac{\partial^2 y_\varepsilon}{\partial x_1^2} -  \varepsilon \sum_{k,l=2}^{3} a_{kl}^{(1)}  \frac{\partial^2 y_\varepsilon}{\partial x_k \partial x_l}
-   \sum_{k,l=2}^{3} \frac{\partial a_{kl}^{(1)}(\xi) }{\partial \xi_k}\Big|_{\xi=\frac{x}{\varepsilon}}  \frac{\partial y_\varepsilon}{\partial x_l}
+ \overrightarrow{V_\varepsilon}^{(1)} \cdot \nabla_x y_\varepsilon
\\
 + \Big( \frac{d v_1^{(1)}(x_1)}{d x_1} + \mathrm{div}_{\bar{\xi}_1}  \overline{V}^{(1)}(x_1, \bar{\xi}_1)\Big|_{\bar{\xi}_1=\frac{\bar{x}_1}{\varepsilon}} +  \lambda\Big) y_\varepsilon =  e^{-\lambda t} F_1  \quad \text{in} \ \ \Omega_\varepsilon^{(1)}\times (0, T),
  \end{multline}
\begin{equation}\label{e_4}
    \big(\mathbb{D}^{(1)}_\varepsilon\nabla y_\varepsilon\big) \cdot \boldsymbol{\nu}_\varepsilon  +  \frac{y_\varepsilon}{\sqrt{1 + \varepsilon^2 |h'_1(x_1)|^2}} \Big(   h'_1(x_1)\, v_1^{(1)} - \overline{V}^{(1)}(x_1, \bar{\xi}_1) \cdot \bar{\nu}_{\bar{\xi}_i}\Big|_{\bar{\xi}_1=\frac{\bar{x}_1}{\varepsilon}}\Big)  =
- e^{-\lambda t} \, \Phi_1  \quad   \text{on} \  \Gamma_\varepsilon^{(1)}\times (0, T),
\end{equation}
where $\overline{V}^{(1)}$, $\boldsymbol{\nu}_\varepsilon$ and $\bar{\nu}_{\bar{\xi}_i}$ are determined in \eqref{V-2D} and \eqref{normal_1} respectively.

Consider any $t_1 \in (0, T).$ Three cases are possible
$$
 1) \  \max_{\overline{\Omega_\varepsilon}\times [0, t_1]} y_\varepsilon  \le  0, \quad
 2) \  0 < \max_{\overline{\Omega_\varepsilon}\times [0, t_1]} y_\varepsilon  \le \max_{\overline{\Gamma_\varepsilon}\times (0, t_1]} y_\varepsilon, \quad
 3) \ 0 < \max_{\overline{\Omega_\varepsilon}\times [0, t_1]} y_\varepsilon  \le y_\varepsilon(x^0, t_0),
 $$
 where the point $P_0 = (x^0, t_0)\in \Omega_\varepsilon \times (0, t_1].$ In the third case, as shown e.g.  in \cite[\S 2, Ch. 1]{Lad_Sol_Ura_1968} the following relations hold at  $P_0$:
\begin{equation}\label{e_5}
  \partial_t y_\varepsilon \ge 0, \qquad \nabla_x y_\varepsilon = \vec{0}, \qquad - \sum_{k,l=1}^{3} a_{kl}^{(i)} \,  \frac{\partial^2 y_\varepsilon}{\partial x_k \partial x_l} \ge 0 \quad \text{for} \ \ i \in \{0, 1, 2, 3\}.
\end{equation}
Therefore, it follows from \eqref{e_1}  and \eqref{e_3}  that
\begin{equation}\label{e_6}
   y_\varepsilon  \le   \sum_{i=0}^{3} \ \max_{\overline{\Omega^{(i)}_\varepsilon}\times [0,T]} |F_i|.
\end{equation}

Similarly, considering the point of the smallest non-positive value of the function $y_\varepsilon,$ which is reached  in
$\Omega_\varepsilon \times (0, t_1],$ we get the estimate
 \begin{equation}\label{e_8}
   \max_{\overline{\Omega_\varepsilon}\times [0, T]} |\psi_\varepsilon|  \le  e^{\lambda T}  \sum_{i=0}^{3} \ \max_{\overline{\Omega^{(i)}_\varepsilon}\times [0,T]} |F_i|.
 \end{equation}

{\bf 2.} Now consider the case when the function $y_\varepsilon$ takes the largest positive value at a  point $P_1$ belonging to $\Gamma_\varepsilon \times (0, t_1].$
The boundary conditions \eqref{e_2} and \eqref{e_4} can be rewritten in the form
\begin{equation}\label{e_8+}
  \sum_{j=1}^{3} b_j^{(0)} \, \frac{\partial y_\varepsilon}{\partial x_j}  =  0  \quad
    \text{on} \ \Gamma_\varepsilon^{(0)}\times (0, T),
\end{equation}
where
$$
b_j^{(0)} :=  \sum_{k=1}^{3} a_{kj}^{(0)}\big(\frac{x}{\varepsilon}\big)\, \nu_k\big(\frac{x}{\varepsilon}\big);
$$
and (we restrict ourselves to the lateral surface $\Gamma_\varepsilon^{(1)}$)
\begin{equation}\label{e_9}
  \sum_{j=1}^{3} b_j^{(1)} \, \frac{\partial y_\varepsilon}{\partial x_j}   +  \frac{y_\varepsilon}{\sqrt{1 + \varepsilon^2 |h'_1(x_1)|^2}} \Big(   h'_1 \, v_1^{(1)} - \overline{V}^{(1)}(x_1, \bar{\xi}_1) \cdot \bar{\nu}_{\bar{\xi}_i}\Big|_{\bar{\xi}_1=\frac{\bar{x}_1}{\varepsilon}}\Big)  =
-  e^{\lambda t} \Phi_1  \quad   \text{on} \  \Gamma_\varepsilon^{(1)}\times (0, T),
\end{equation}
where
$$
b_1^{(1)} := - \varepsilon a_{11}^{(1)} \frac{h'_1(x_1)}{\sqrt{1 + \varepsilon^2 |h'_1(x_1)|^2}}, \qquad
b_j^{(1)} :=  \frac{1}{\sqrt{1 + \varepsilon^2 |h'_1(x_1)|^2}} \sum_{k=2}^{3} a_{kj}^{(1)}\big(\frac{\bar{x}_1}{\varepsilon}\big) \, \nu_k\big(\frac{\bar{x}_1}{\varepsilon}\big), \ \ j\in \{2, 3\}.
$$
It is easy to verify that due to \eqref{n1}
\begin{equation}\label{e_10}
  \big(b_1^{(0)}, b_2^{(0)}, b_3^{(0)}\big) \cdot \boldsymbol{\nu}_\varepsilon = \sum_{j, k=1}^{3} a_{kj}^{(0)}\big(\frac{x}{\varepsilon}\big)\, \nu_k\big(\frac{x}{\varepsilon}\big) \, \nu_j\big(\frac{x}{\varepsilon}\big) \ge \kappa_0^{(0)} > 0,
\end{equation}
\begin{equation}\label{e_11}
  \big(b_1^{(1)}, b_2^{(1)}, b_3^{(1)}\big) \cdot \boldsymbol{\nu}_\varepsilon =
    a_{11}^{(1)} \frac{\varepsilon^2  |h'_1(x_1)|^2}{1 + \varepsilon^2 |h'_1(x_1)|^2}
   + \frac{1}{1 + \varepsilon^2 |h'_1(x_1)|^2}\sum_{j, k=2}^{3} a_{kj}^{(1)}\big(\frac{\bar{x}_1}{\varepsilon}\big) \, \nu_k\big(\frac{\bar{x}_1}{\varepsilon}\big) \, \nu_j\big(\frac{\bar{x}_1}{\varepsilon}\big) \ge \frac{\kappa_0^{(1)}}{c_1} > 0.
\end{equation}

Based on the statement of Theorem 7 (\!\cite[\S 1]{Il_Ole_Kal_1962}), the inequalities \eqref{e_10} and \eqref{e_11} mean that
\begin{equation}\label{e_12}
  \sum_{j=1}^{3} b_j^{(i)} \, \frac{\partial y_\varepsilon}{\partial x_j}\Big|_{P_1} > 0,
\end{equation}
where the value of the index $i$ depends on the position of $P_1.$
Thus, the point $P_1$ cannot  lie both  on $\Gamma_\varepsilon^{(0)}\times (0, T)$ and circular strips on the lateral surfaces of the thin cylinders near their bases
$\{\Upsilon_{\varepsilon}^{(i)} (\ell_i)\times (0, T)\}_{i=1}^3$ and $\{\Upsilon_{\varepsilon}^{(i)} (\varepsilon\,\ell_0)\times (0, T)\}_{i=1}^3,$ where $\{h'_i\},$ $\{\overline{V}^{(i)}\},$  $\{\Phi_i\}$ vanish and where this derivative is equal to zero.

Next, for definiteness, we assume that $P_1 \in \Gamma_\varepsilon^{(1)}\times (0, T),$ and to simplify calculations and to see a clearer dependence on  $\varepsilon,$ consider the case when
$\mathbb{D}^{(1)}_\varepsilon $ is  the identity matrix, i.e., $\mathrm{div}_x \big( \mathbb{D}^{(1)}_\varepsilon \nabla_x y_\varepsilon\big) =
\Delta_x y_\varepsilon .$

We introduce a new function $\eta_\varepsilon$ using the following substitution $y_\varepsilon = \eta_\varepsilon \, z_\varepsilon,$ where
\begin{equation*}
  z_\varepsilon(x) = 1 + \varepsilon\, \frac{m\, h_1(x_1)}{2 } \Big( 1 - \Big(\frac{x_2}{\varepsilon h_1(x_1)}\Big)^2 - \Big(\frac{x_3}{\varepsilon h_1(x_1)}\Big)^2\Big)
\end{equation*}
and the constant
\begin{equation}\label{m_1}
  m = 1 +  \max_{\overline{\Omega^{(1)}_\varepsilon}} \Big(|h'_1(x_1)| \, |v_1^{(1)}(x_1)| + \big|\overline{V}^{(1)}(x_1,\frac{\bar{x}_1}{\varepsilon})\big| \Big) .
\end{equation}
It is easy to verify that
\begin{equation}\label{z_1}
  z_\varepsilon\big|_{\Gamma^{(1)}_\varepsilon} =1, \qquad - \partial_{\bar{\nu}} z_\varepsilon\big|_{\Upsilon^{(1)}_\varepsilon(x_1)} = m, \qquad
  1\le z_\varepsilon \le 1  +  \tfrac{1}{2 }\, \varepsilon \, m\, h_1(x_1) \quad \text{in} \ \ \Omega^{(1)}_\varepsilon ,
  \end{equation}
  \begin{equation}\label{z_2}
    |\partial_{x_1} z_\varepsilon| = \Big| \frac{m\, h'_1(x_1)}{2} \Big(\varepsilon + \frac{x_2^2 + x_3^2}{\varepsilon\, h^2_1(x_1)} \Big) \Big| \le \varepsilon\, m\, |h'_1(x_1)|, \qquad |\nabla_{\bar{x}_1 } z_\varepsilon|^2 = \frac{m^2 |\bar{x}_1|^2}{\varepsilon^2 h^2_1(x_1)} \le m^2,
  \end{equation}
  \begin{equation}\label{z_3}
  |\partial^2_{x_1 x_1} z_\varepsilon| \le \varepsilon\, m \Big( |h''_1(x_1)| + \frac{|h'_1(x_1)|^2}{h_1(x_1)} \Big), \qquad
  |\Delta_{\bar{x}_1} z_\varepsilon | =  \varepsilon^{-1} \frac{2m}{h_1(x_1)} .
  \end{equation}

In general case, we should take a function
$$
  z_\varepsilon(x) = 1 + \varepsilon\, \varrho\big(\frac{\bar{x}_1}{\varepsilon\, h_1(x_1)}\big),
$$
where $\varrho$ is any function from  $C^2(\overline{B_1(0)})$  such that
$$
\varrho\big|_{\partial B_1(0)} = 0, \qquad - \big(\tilde{\mathbb{D}}^{(1)}_{h_1(x_1) \bar{\zeta}_1} \nabla_{\bar{\zeta}_1} \varrho\big)\cdot
\bar{\nu}_{\bar{\zeta}_1}\big|_{\partial B_1(0)} = m .
$$
Here $\bar{\nu}_{\bar{\zeta}_1}$ is the outward unit normal to the boundary of the
disk $B_1(0):=\{\bar{\zeta}_1=(\zeta_2, \zeta_3)\in \Bbb R^2\colon |\bar{\zeta}_1| < 1\},$ the matrix $\tilde{\mathbb{D}}^{(1)}$ is determined in \eqref{mat-2D}. Moreover,  estimates of the same order with respect to $\varepsilon$
as in \eqref{z_1} --  \eqref{z_3} hold.

After the substitution, the function $\eta_\varepsilon$ satisfies the differential equation
\begin{multline}\label{eta_1}
  \partial_t \eta_\varepsilon - \varepsilon\, \Delta \eta_\varepsilon + \Big(\frac{2 \varepsilon\, \nabla z_\varepsilon}{z_\varepsilon} + \overrightarrow{V_\varepsilon}^{(1)}\Big)\cdot \nabla \eta_\varepsilon
  \\
  +
  \Big(\frac{\varepsilon\, \Delta  z_\varepsilon}{z_\varepsilon} - \frac{2 \varepsilon\, |\nabla z_\varepsilon|^2}{z^2_\varepsilon}
    - \frac{ v_1^{(1)}\, \partial_{x_1} z_\varepsilon + \varepsilon^2 \, \overline{V}^{(1)}\cdot \nabla_{\bar{x}_1} z_\varepsilon}{z_\varepsilon}
   + \frac{d v_1^{(1)}}{d x_1} + \mathrm{div}_{\bar{\xi}_1}  \overline{V}^{(1)}(x_1, \bar{\xi}_1)\Big|_{\bar{\xi}_1=\frac{\bar{x}_1}{\varepsilon}} +  \lambda   \Big) \eta_\varepsilon = z_\varepsilon \, e^{-\lambda t}\, F_1
\end{multline}
in $\Omega^{(1)}_\varepsilon\times (0, T)$, and the boundary condition
\begin{equation}\label{eta_2}
  \sum_{j=1}^{3} b_j^{(1)} \, \frac{\partial \eta_\varepsilon}{\partial x_j}   +   \frac{\eta_\varepsilon}{\sqrt{1 + \varepsilon^2 |h'_1(x_1)|^2}} \Big( m  + \varepsilon \, m \, |h'_1(x_1)|^2 +  h'_1(x_1)\, v_1^{(1)}(x_1) -   \overline{V}^{(1)}(x_1, \bar{\xi}_1) \cdot \bar{\nu}_{\bar{\xi}_i}\Big|_{\bar{\xi}_1=\frac{\bar{x}_1}{\varepsilon}}\Big)  =
-  e^{-\lambda t}\, z_\varepsilon\, \Phi_1
\end{equation}
on $ \Gamma_\varepsilon^{(1)}\times (0, T),$ where now
$$
b_1^{(1)} := - \varepsilon  \frac{h'_1(x_1)}{\sqrt{1 + \varepsilon^2 |h'_1(x_1)|^2}}, \qquad
b_j^{(1)} :=  \frac{1}{\sqrt{1 + \varepsilon^2 |h'_1(x_1)|^2}}  \, \nu_j\big(\frac{\bar{x}_1}{\varepsilon}\big), \ \ j\in \{2, 3\}.
$$

Due to \eqref{z_1} --  \eqref{z_3}  we can take the constant $\lambda$ independently of $\varepsilon$ in such a way that the coefficient at $\eta_\varepsilon$ in \eqref{eta_1} is bounded below by a positive constant.  And thanks to our  choice  of the constant $m$ (see \eqref{m_1}),  for all $\varepsilon \in (0, 1)$ and all $x_1 \in [0, \ell_1]$
$$
m  + \varepsilon \, m \, |h'_1(x_1)|^2 +  h'_1(x_1)\, v_1^{(1)}(x_1) -   \overline{V}^{(1)}(x_1, \bar{\xi}_1) \cdot \bar{\nu}_{\bar{\xi}_i}\Big|_{\bar{\xi}_1=\frac{\bar{x}_1}{\varepsilon}} \ge 1.
$$
In virtue of the first and third relations in \eqref{z_1} the function $\eta_\varepsilon$ takes also the largest positive value at the  point $P_1$. Therefore, it follows from \eqref{e_12} and \eqref{eta_2} that
$$
 \eta_\varepsilon(P_1) \le - \Big(\sqrt{1 + \varepsilon^2 |h'_1(x_1)|^2} \, z_\varepsilon\,  e^{-\lambda t}\,  \Phi_1\Big)\Big|_{P_1}.
$$

Similarly, we consider the case when the function $y_\varepsilon$ reaches its smallest non-positive value  at a point
$P_2 \in \Gamma^{(1)}_\varepsilon \times (0, t_1].$ Since
$\sum_{j=1}^{3} b_j^{(1)} \, \partial_{x_j} y_\varepsilon \big|_{P_2} \le 0,$  we get the opposite inequality
$$
 \eta_\varepsilon(P_2) \ge - \Big(\sqrt{1 + \varepsilon^2 |h'_1(x_1)|^2} \, z_\varepsilon\,  e^{-\lambda t}\,  \Phi_1\Big)\Big|_{P_2}.
$$
Thus, in the second case
$$
\max_{\overline{\Omega_\varepsilon}\times [0, t_1]} |y_\varepsilon|  \le \max_{\overline{\Gamma^{(1)}_\varepsilon}\times (0, t_1]} |y_\varepsilon|
\le C_1 \, \max_{\overline{\Gamma^{(1)}_\varepsilon}\times (0, T]} |\Phi_1 (x, t)|.
$$
The lemma is proved.
\end{proof}

Let us now prove an a priori estimate for the gradient of the solution in $L^2$-norm.

\begin{corollary}\label{apriory_estimate} For the solution $\psi_\varepsilon$  to the problem \eqref{probl+} the estimate
\begin{equation}\label{appriory_estimate}
\| {\nabla_x \psi_\varepsilon}\|_{L^2(\Omega_\varepsilon\times (0, T))} \le
 C_2(T) \bigg(\sum_{i=0}^{3}\max_{\overline{\Omega^{(i)}_\varepsilon}\times [0, T]} |F_i| +  \sum_{i=1}^{3}\max_{\overline{\Gamma^{(i)}_\varepsilon}\times [0, T]} |\Phi_i| \bigg) ,
  \end{equation}
is satisfied,  where the constant $C_2$ does not depend on  $\varepsilon, \, \psi_\varepsilon, \, \{F_i\}, \, \{\Phi_i\}.$
\end{corollary}

\begin{proof} We multiply the differential equations of the problem \eqref{probl+} with  $\psi_\varepsilon$ and integrate them over
$\Omega_\varepsilon \times (0, \tau),$ where $\tau$ is an arbitrary number from $(0, T).$ Integrating by parts and taking into account the boundary conditions and initial condition, we get
 \begin{multline*}
  \frac{1}{2} \int\limits_{\Omega_\varepsilon}\psi_\varepsilon^2\big|_{t=\tau}\, dx + \varepsilon\int\limits_{0}^{\tau} \int \limits_{\Omega_\varepsilon}
    (\mathbb{D}_\varepsilon {\nabla \psi_\varepsilon}) \cdot {\nabla \psi_\varepsilon} \, dxdt  -  \int\limits_{0}^{\tau} \int \limits_{\Omega_\varepsilon} \psi_\varepsilon\,
    \overrightarrow{V_\varepsilon} \cdot
    {\nabla \psi_\varepsilon} \, dx dt
    \\
    =
    \sum\limits_{i=0}^{3} \
    \int\limits_{0}^{\tau}\int \limits_{\Omega_\varepsilon^{(i)}} F_i \, \psi_\varepsilon \, dx dt
 -  \varepsilon \sum \limits_{i=1}^{3} \
    \int\limits_{0}^{\tau}\int \limits_{\Gamma_\varepsilon^{(i)}} \Phi_i \, \psi_\varepsilon \, dS_x dt.
  \end{multline*}
 In view of the ellipticity conditions \eqref{n1} and the Cauchy-Bunyakovsky-Schwarz inequality,  we have
\begin{multline}\label{s_1}
\varepsilon \sum_{i=0}^{3} \kappa_0^{(i)} \| {\nabla \psi_\varepsilon}\|^2_{L^2(\Omega^{(i)}_\varepsilon\times (0, \tau))}
-   \int\limits_{0}^{\tau} \int \limits_{\Omega_\varepsilon} \psi_\varepsilon\,
    \overrightarrow{V_\varepsilon} \cdot
    {\nabla \psi_\varepsilon} \, dx dt
  \\
  \le
     \sum_{i=0}^{3}
    \|F_i\|_{L^2(\Omega_\varepsilon^{(i)}\times (0, \tau))}  \|\psi_\varepsilon\|_{L^2(\Omega_\varepsilon^{(i)}\times (0, \tau))}
    +
     \varepsilon
  \sum_{i=1}^{3} \|\Phi_i\|_{L^2(\Gamma_\varepsilon^{(i)}\times (0, \tau))}  \|\psi_\varepsilon\|_{L^2(\Gamma_\varepsilon^{(i)}\times (0, \tau))}.
   \end{multline}

With the help of the inequalities
\begin{equation}\label{p_0}
\|\psi_\varepsilon \|_{L^2(\Omega_\varepsilon^{(0)})} \le C_1 \sqrt{\varepsilon} \|\nabla \psi_\varepsilon \|_{L^2(\Omega_\varepsilon)}, \qquad
\|\psi_\varepsilon \|_{L^2(\Omega_\varepsilon^{(i)})} \le C_1  \|\nabla \psi_\varepsilon \|_{L^2(\Omega_\varepsilon^{(i)})}
\end{equation}
(this is due to  the uniform Dirichlet conditions on the bases $\{\Upsilon_{\varepsilon}^{(i)}(\ell_i)\}_{i=1}^3$, for more detail see \cite[Sec.~2]{M-AA-2021})
and  the inequality
\begin{equation}\label{p_1}
 \varepsilon \int_{\Gamma_\varepsilon^{(i)}} v^2 \, d\sigma_x
 \leq C_2 \Bigg( \varepsilon^2 \int_{\Omega_\varepsilon^{(i)}} |\nabla_{x}v|^2 \, dx
 +    \int_{\Omega_\varepsilon^{(i)}} v^2 \, dx \Bigg),  \quad \forall\, v \in H^1(\Omega_\varepsilon^{(i)} \ \ (i\in \{1, 2, 3\}),
\end{equation}
 proved in \cite{M-MMAS-2008}, we obtain from \eqref{s_1} that
 \begin{multline}\label{s_2}
 \varepsilon \sum_{i=0}^{3} \kappa_0^{(i)} \| {\nabla \psi_\varepsilon}\|^2_{L^2(\Omega^{(i)}_\varepsilon\times (0, \tau))}
    \le  C_3 \sqrt{T}\, \varepsilon \, \max_{\overline{\Omega_\varepsilon}\times [0, T]} |\psi_\varepsilon| \, \max_{\overline{\Omega_\varepsilon}} |\overrightarrow{V_\varepsilon}|\,    \|\nabla \psi_\varepsilon\|_{L^2(\Omega_\varepsilon \times (0, \tau))}
    \\
    + C_3 \, \bigg(\sum_{i=1}^{3}  \Big(\|F_i\|_{L^2(\Omega_\varepsilon^{(i)}\times (0, \tau))}  +
    \sqrt{\varepsilon}\, \|\Phi_i\|_{L^2(\Gamma_\varepsilon^{(i)}\times (0, \tau))}\Big)  + \sqrt{\varepsilon}\, \|F_0\|_{L^2(\Omega_\varepsilon^{(0)}\times (0, \tau))}\bigg)   \|\nabla \psi_\varepsilon\|_{L^2(\Omega_\varepsilon \times (0, \tau))}.
 \end{multline}
Using \eqref{max_1}, from  \eqref{s_2}   it follows \eqref{appriory_estimate}.
\end{proof}

The main results follow directly from Proposition~\ref{Prop-3-1}, Lemma~\ref{the maximum principle} and Corollary~\ref{apriory_estimate}.
\begin{theorem}\label{Th_1}
Let the assumptions made in Section~\ref{Sec:Statement}  and Remark~\ref{differ}  hold. Then the series \eqref{asymp_expansion} is the asymptotic expansion for the solution $u_\varepsilon$  to the problem~\eqref{probl} both in the Banach space
  $C(\overline{\Omega}_\varepsilon\times [0,T])$ and the Sobolev space $L^2\big((0,T); H^1(\Omega_\varepsilon)\big);$ and
  for  any $M\in \Bbb N$ there exist $C_M>0$ and $\varepsilon_0>0$ such that for all $\varepsilon\in(0, \varepsilon_0)$ the  estimates
 \begin{equation}\label{main_1}
   \max_{\overline{\Omega_\varepsilon}\times [0, T]} |u_\varepsilon - \mathcal{U}_{M-1}|  \le C_M \, \varepsilon^{M-\gamma}
 \end{equation}
and
\begin{equation}\label{main_2}
\|\nabla_x(u_\varepsilon - \mathcal{U}_{M-1})\|_{L^2(\Omega_\varepsilon\times (0, T))}
    \le C_M \, \varepsilon^{M-1-\gamma}
 \end{equation}

 \medskip
 \noindent
 are satisfied,  where $\mathcal{U}_{M-1}$ is the partial sum of the series \eqref{asymp_expansion} determined in \eqref{aaN}  and $\gamma$ is a fixed number from the interval $(\frac23, 1).$
\end{theorem}
\begin{proof}
 In fact, from Proposition~\ref{Prop-3-1},  Lemma~\ref{the maximum principle} and Corollary~\ref{apriory_estimate} we get the estimates
 \begin{equation}\label{main_1+}
   \max_{\overline{\Omega_\varepsilon}\times [0, T]} |u_\varepsilon - \mathcal{U}_{M}|  \le C_M \, \varepsilon^{M-\gamma},
 \end{equation}
\begin{equation}\label{main_2+}
\|\nabla_x(u_\varepsilon - \mathcal{U}_{M})\|_{L^2(\Omega_\varepsilon\times (0, T))}
    \le C_M \, \varepsilon^{M-\gamma}.
 \end{equation}
Then, from \eqref{main_1+} we deduce
$$
\max_{\overline{\Omega_\varepsilon}\times [0, T]} |u_\varepsilon - \mathcal{U}_{M-1}| \le
\max_{\overline{\Omega_\varepsilon}\times [0, T]} |u_\varepsilon - \mathcal{U}_{M}|
     + \varepsilon^M \Big(\max_{\overline{\Omega_\varepsilon}\times [0, T]} |\widehat{u}_M| + \max_{\overline{\Omega_\varepsilon}\times [0, T]} |\widehat{N}_M| +
\max_{\overline{\Omega_\varepsilon}\times [0, T]} |\widehat{\Pi}_M|\Big)
\le C_M \, \varepsilon^{M-\gamma},
$$
with a new constant $C_M;$ here the coefficients $\widehat{u}_M,$ $\widehat{N}_M,$ and $\widehat{\Pi}_M$ are determined in \eqref{asymp_expansion}.

Similarly, we get \eqref{main_2} from \eqref{main_2+}, but here we need to take into account the volume of the domains over which we integrate. The estimates \eqref{main_1} and \eqref{main_2} obtained for any $M\in \Bbb N$ mean that  the series \eqref{asymp_expansion} is the asymptotic expansion for the solution $u_\varepsilon$ in the spaces mentioned above.
\end{proof}

It is clear that for applied problems there is no need to construct a complete asymptotic expansion for the solution, but it is enough to take an approximation of the solution with the required accuracy. Then weaker assumptions about the smoothness of the coefficients and given functions can be also expected.

For example, if $M=1$ in \eqref{main_1} and \eqref{main_2}, then we obtain the corresponding estimates for the difference between the solution $u_\varepsilon$ and the zero-order approximation
\begin{equation}\label{zero_app}
 \mathcal{U}_0 =
\left\{
  \begin{array}{ll}
    w_0^{(i)} (x_i, t), & x \in \Omega^{(i)}_\varepsilon \cap \{x\colon x_i\in [3\ell_0 \varepsilon^\gamma, \ell_i]\}, \ \ i\in\{1, 2\},
\\[4pt]
    w_0^{(3)} (x_3, t) + \chi_\delta^{(3)} (x_3) \, \Pi_0^{(3)}\big(\frac{\ell_3 - x_3}{\varepsilon}, t \big), & x \in \Omega^{(3)}_\varepsilon \cap \{x\colon x_3\in [3\ell_0 \varepsilon^\gamma, \ell_3]\},
 \\
    N_0\big(\frac{x}{\varepsilon}, t\big), &x \in \Omega^{(0)}_\varepsilon \cup \Big(\bigcup\limits_{i=1}^3 \Omega^{(i)}_\varepsilon \cap \{x\colon x_i\in [\varepsilon \ell_0,  2\ell_0 \varepsilon^\gamma]\}\Big),
\\
\chi_{\ell_0}^{(i)}\left(\frac{x_i}{\varepsilon^\gamma}\right)  w_0^{(i)} (x_i, t) +
\big(1- \chi_{\ell_0}^{(i)}\left(\frac{x_i}{\varepsilon^\gamma}\right)\big) N_0\big(\frac{x}{\varepsilon}, t\big), &x \in
\Omega^{(i)}_\varepsilon \cap \{x\colon x_i\in [2\ell_0 \varepsilon^\gamma, 3\ell_0 \varepsilon^\gamma]\}, \ \ i\in\{1, 2,3\},
  \end{array}
\right.
\end{equation}
where $\{w_0^{(i)}\}_{i=1}^3$ form the solution to the limit problem \eqref{limit_prob} and they are represented by formulas \eqref{model_solution_w_0}, the coefficient $\Pi_0^{(3)}$ is determined in \eqref{Pi_0}, and $N_0$ is the solution to problem \eqref{N_0_prob},
the cut-off functions $\chi_\delta^{(3)}, \  \chi_{\ell_ 0}^{(i)}$ are  defined by formulas \eqref{cut-off_functions}, the variable $t\in [0, T].$

As follows from the proof of Theorem~\ref{Th_1}, to obtain the estimate
\begin{equation}\label{main_3}
   \max_{\overline{\Omega_\varepsilon}\times [0, T]} |u_\varepsilon - \mathcal{U}_{0}|  \le C_1 \, \varepsilon^{1-\gamma},
 \end{equation}
we should construct the partial sum $\mathcal{U}_{1}.$ This means that $\{w_1^{(i)}\}_{i=1}^3$ must be classical solutions to the corresponding problems (see \eqref{prob_w_1}) and $\{w_0^{(i)}\}_{i=1}^3$ must then have the $C^2$-smoothness. Therefore, the following statement holds.
\begin{corollary}\label{Coro_5_2}
 Let, in addition to the  assumptions {\bf A1} -- {\bf A3}, the following ones are satisfied:
$$
v_i^{(i)} \in C^2([0, \ell_i]), \quad   q''_i(0)=0, \quad  \partial_t \varphi^{(i)}\big|_{t=0}=0 \quad \text{for} \ \  i\in \{ 1,2,3\},
$$
and $\varphi^{(0)}$ belongs to $C^2$ in $t\in [0,T]$ and $\partial_t \varphi^{(0)}\big|_{t=0}=\partial^2_{tt} \varphi^{(0)}\big|_{t=0}=0.$
Then the inequality \eqref{main_3} hold.
\end{corollary}

Error estimates in the $L^2$-norm for thin domains must be written in the rescaled form, namely, they must be divided by the volume of the corresponding thin domain. Therefore, the first appropriate approximation for the gradient of the solution is as follows
\begin{equation}\label{grad_1}
\tfrac{1}{\sqrt{|\Omega_\varepsilon|}}\, \|\nabla_x(u_\varepsilon - \mathcal{U}_{1})\|_{L^2(\Omega_\varepsilon\times (0, T))}
    \le C_1 \, \varepsilon^{1-\gamma},
\end{equation}
where $|\Omega_\varepsilon|$ is  the Lebesque measure of $\Omega_\varepsilon,$ which  obviously has the order $\varepsilon^2.$
Taking the argumentations before Corollary~\ref{Coro_5_2} into account, we conclude the following.
\begin{corollary}\label{Coro_5_3}
 Let, in addition to the  assumptions {\bf A1} -- {\bf A3}, the following ones are satisfied:
$$
v_i^{(i)} \in C^3([0, \ell_i]), \quad   q''_i(0)=q'''_i(0)=0, \quad  \partial_t \varphi^{(i)}\big|_{t=0}=\partial^2_{tt} \varphi^{(i)}\big|_{t=0}=0 \quad \text{for} \ \  i\in \{ 1,2,3\},
$$
and $\varphi^{(0)}$ belongs to $C^3$ in $t\in [0,T]$ and $\partial_t \varphi^{(0)}\big|_{t=0}=\partial^2_{tt} \varphi^{(0)}\big|_{t=0}=\partial^3_{ttt} \varphi^{(0)}\big|_{t=0}=0.$

Then the inequality \eqref{grad_1} hold.
\end{corollary}


\section{Conclusions}

{\bf 1.}   In the introduction, we have outlined the problems of purely one-dimensional, possibly  nonlinear transport modelling on graphs which root in the  inherently incomplete understanding of multidimensional effects at graph's vertices. Our approach  overcomes (in the simpler linear case) these problems by a mixed-dimensional modelling.
Indeed, even  our  zeroth-order approximation \eqref{zero_app}  contains  the coefficient $N_0$ that is the solution to the three-dimensional problem \eqref{N_0_prob} at the node, which takes into account the local geometric irregularity of the node,  physical processes inside, and  the diffusion flow through the boundary condition
\begin{equation}\label{6_1}
  -\sigma_\xi(N_0(\xi,t))  =  \varphi^{(0)}(\xi,t)
\end{equation}
at the node boundary $\Gamma_0,$ and the coefficient  $\Pi_0$ that is the solution to the boundary-layer problem~\eqref{prim+probl+0} near the base of the cylinder $\Omega_\varepsilon^{(3)}.$ In addition, the impact of the given function $\varphi^{(0)}$
is also manifested in the Kirchhoff transmission condition
\begin{equation}\label{6_2}
  \sum_{i=1}^{3}  h_i^2(0) \,\mathrm{v}_i \,  w_0^{(i)}(0, t) = \breve{\varphi}^{(0)}(t)
\end{equation}
in the limit problem \eqref{limit_prob}, where $\breve{\varphi}^{(0)}$ is determined in \eqref{breve_phi}. If $\breve{\varphi}^{(0)}$ vanishes, then we have the classical  Kirchhoff's junction rule stating that the sum of fluxes flowing into a node is equal to the sum of fluxes leaving that node.

What can be said about the influence of  the functions $\{\varphi^{(i)}\big(\tfrac{\overline{x}_i}{\varepsilon}, x_i, t\big)\}_{i=1}^3$ that are right-hand sides in the boundary conditions on the lateral surfaces of the  thin cylinders in the original problem \eqref{probl}?
Despite the fact that they  are multiplied by the small parameter $\varepsilon,$ their influence is observed in the corresponding
one-dimensional balance equation
\begin{equation}\label{6_3}
  h^2_i(x_i)\, \partial_t{w}^{(i)}_0 + \partial_{x_i}\big( v_i^{(i)}(x_i)\,  h^2_i(x_i)\, w^{(i)}_0 \big) = - 2 h_i(x_i)\,  \widehat{\varphi}^{(i)}(x_i, t)
\end{equation}
in the limit problem \eqref{limit_prob}, and the function $\widehat{\varphi}^{(i)},$ determined in \eqref{hat_phi},  arises from the diffusion flow and convection flow across the lateral surface of the thin cylinder $\Omega^{(i)}_\varepsilon.$
Balance equations of this kind arise in many areas of natural science, typically as equations governing      momentum or
energy. 

Also in \eqref{6_3}, we observe the function $h_i$, which describes the aperture  of the thin cylinder $\Omega^{(i)}_\varepsilon,$ and the function $v_i^{(i)}$, which is the longitudinal component of the velocity field $\overrightarrow{V_\varepsilon} ^{(i)}$ in this cylinder.
The direct influence of the diffusion matrix and non-longitudinal components of $\overrightarrow{V_\varepsilon} ^{(i)}$  reveals itself in the definition of the next terms of the regular asymptotics, namely, the coefficient $u_1^{(i)}$ (see \eqref{eq_1} and \eqref{bc_1}) and coefficient
$w_1^{(i)}$ (see \eqref{prob_w_1}).

Thus, our  asymptotic approach is a very general and unified algorithm for finding the asymptotic approximation with prescribed  accuracy for solutions to convection-dominated problems in thin graph-like junctions. The mixed-dimensional ansatz allows us to  account for various factors in statements of boundary-value problems on graphs.

\medskip

{\bf 2.}  The second case  of one inlet and two outlets $(v^{(1)}_1 < 0,\  v^{(2)}_2 >  0, \  v^{(3)}_3 > 0 )$ in Section \ref{par_222} is particularly interesting because the Kirchhoff conditions do not suffice to ensure wellposedness. Our choice of the boundary condition \eqref{cont_condition}  is justified as follows. If the function $\breve{\varphi}^{(0)}$ is equal to zero, then thanks to~\eqref{cond_1} we get the continuity condition
\begin{equation}\label{6_4}
w^{(1)}_0\big|_{x_1=0} = w^{(2)}_0\big|_{x_2=0} = w^{(3)}_0\big|_{x_3=0}
\end{equation}
which is lost in the limit for convection-dominated problems in thin graph-like junctions. If $\breve{\varphi}^{(0)} \neq 0$ the continuity condition is not satisfied, just as in the first case  $(v^{(1)}_1 < 0,\  v^{(2)}_2 <  0, \  v^{(3)}_3 > 0 ).$
From a physical point of view, this fact can be explained as the duality of the description of phenomena using first-order hyperbolic equations: this is the propagation of waves on the one hand, and the movement of classical particles on the other hand (wave-particle duality).

Therefore, the presence of the coefficient $N_0$ in the zero-order approximation \eqref{zero_app} is very important, since it neutralizes this dualism and ensures the continuity of the approximation in a vicinity of the node. In the second case and if ${\varphi}^{(0)} = 0$ and
\eqref{6_4} hold, it follows from \eqref{N_0_prob} that $N_0 = w^{(1)}_0(0,t)$ and
the zeroth-order approximation is significantly simplified, respectively
(see \eqref{zero_app} and plug in $N_0 = w^{(1)}_0(0,t)$).

The coefficients $\{N_k\}$ of the node-layer asymptotics are in fact analogues of the boundary-layer.
This is in full agreement with the asymptotic theory of boundary value problems with a small parameter at higher derivatives (in the limit, such problems become problems of a lower order). Solutions to such problems have local features: they are  bounded, but in some part of the region (usually near the boundary) their derivatives are very large. In our case, the solution has such local singularities both near the bases of some thin cylinders and in a neighborhood of the node. The difference is that $\{N_k\}$  have polynomial growth at infinity (see \eqref{N_0_prob} and \eqref{N_k_prob}), and because of this we had to introduce special cut-off functions to cancel out this growth.

Moreover, the choice of the boundary condition
$w_k^{(2)}\big|_{x_2 =0} =  w_k^{(1)}(0,t)$ for the other terms in the second case reduces the absolute value of the partial derivatives of  $\{N_k\}$ with respect to the spatial variables (as we saw above, if ${\varphi}^{(0)} = 0,$ then all derivatives of $N_0$ vanish), and as a result we will have minimal constants in the right-hand sides of the asymptotic estimates \eqref{main_1}, \eqref{main_2}, \eqref{main_3} and \eqref{grad_1}.

\medskip

{\bf 3.} This approach of ours can be applied without any changes to convection-dominated problems in thin graph-like junctions that have $n$ incoming thin cylinders with respect to the velocity field $\overrightarrow{V_\varepsilon}$ and one outgoing or one incoming cylinder and $n$ outgoing. Such junctions are used, for example, in modeling blood flow in veins or arteries, and transport processes in root systems or tree systems.

In the general case, when $n$ thin cylinders enter the node and $m$ exit, the question is open. As we noted in the introduction, this is also a big problem for conservation laws on networks.  In our opinion, based on our comments in \textbf{2.}, one should choose such boundary conditions for the outgoing solutions $\{w_0^{(i)}\}$ in order to minimize the absolute values of the derivatives of $N_0.$ This question will be the subject of a future article.

We are also interested in studying the influence on the solution asymptotics
by the intensity factors $\varepsilon^\alpha,$ $\varepsilon^\beta,$ and $\varepsilon^\delta$. These are  present both in the velocity field
$$
\overrightarrow{V_\varepsilon}^{(1)} = \left(v^{(1)}_1(x_1), \ \varepsilon^\alpha\,  v^{(1)}_2(x_1, \tfrac{\overline{x}_1}{\varepsilon}), \
 \varepsilon^\alpha \, v^{(1)}_3(x_1,\tfrac{\overline{x}_1}{\varepsilon})
\right)
$$
and in the boundary conditions
$$
\big(-  \varepsilon \,  \mathbb{D}_\varepsilon\nabla_x u_\varepsilon +  u_\varepsilon \, \overrightarrow{V}_\varepsilon\big) \cdot \boldsymbol{\nu}_\varepsilon  =  \varepsilon^\beta \varphi^{(i)}_\varepsilon \quad
   \text{on} \ \  \Gamma^{(i)}_\varepsilon, \qquad
\big(-  \varepsilon \,  \mathbb{D}_\varepsilon\nabla_x u_\varepsilon\big) \cdot \boldsymbol{\nu}_\varepsilon  = \varepsilon^\delta \varphi^{(0)}_\varepsilon \quad  \text{on} \ \Gamma^{(0)}_\varepsilon.
$$
It follows from our results that if $\alpha >1,$ $\beta >1,$ and $\delta > 0,$ then the zero-order approximation will be equal to \eqref{zero_app}.
But what happens if $\alpha < 1$ or $\beta < 1$ or $\delta < 0,$ is  still unknown. For other boundary-value problems, such studies were carried out in \cite{Mel_Kle_2019,M-AA-2021}.

In our studies, we significantly used the assumption that the vector field $\overrightarrow{V_\varepsilon}$ is conservative at the node. It  would be interesting to obtain some results without this assumption.

\section*{Acknowledgements}
The first author is grateful to the Alexander von Humboldt Foundation for the possibility to carry
out this research at the University of Stuttgart. The second author thanks for  funding by the  Deutsche Forschungsgemeinschaft (DFG, German
Research Foundation) – Project Number 327154368 – SFB 1313.

\end{document}